\documentclass[a4paper,12pt]{article}

\usepackage{hyperref}

\usepackage{fullpage}
\usepackage{graphicx}
\usepackage{amsmath}
\usepackage{amssymb}
\usepackage{subfigure}
\usepackage{epsfig}

\usepackage{caption,color} 
\usepackage{amsthm}

\usepackage{bm}

\usepackage{url}

\usepackage{accents}

\theoremstyle{plain}
\newtheorem{thm}{Theorem}[section]
\newtheorem{lem}[thm]{Lemma}
\newtheorem{prop}[thm]{Proposition}

\theoremstyle{remark}

\author{
Li-Chang Hung\footnote{Corresponding author's email address: \texttt{lichang.hung@gmail.com
} 
         }
\vspace{10mm}
\\
 \small
 \textit{
 Department of Mathematics, National Taiwan University, Republic of Taiwan}
\\
 \small 
}

\title{An N-barrier maximum principle for elliptic systems arising from the study of traveling waves in reaction-diffusion systems
}

\date{
\small 
L.-C. Hung dedicates the N-Barrier Maximum principle to NBM 
}

\begin{document}
\maketitle

\begin{abstract}

By employing the N-barrier method developed in the paper, we establish a new N-barrier maximum principle for diffusive Lotka-Volterra systems of two competing species. As an application of this maximum principle, we show under certain conditions, the existence and nonexistence of traveling waves solutions for systems of three competing species. In addition, new $(1,0,0)$-$(u^{\ast},v^{\ast},0)$ waves are given in terms of the tanh function provided that the parameters satisfy certain conditions.










\end{abstract}



\section{Introduction}
\vspace{5mm}

Species diversity refers to the number of different species and abundance of each species that live within an ecological system. To be more specific, species diversity takes into consideration species richness and species evenness; the former is defined as the total number of different species and the latter the variation of abundance in individuals per species. The importance of species diversity to an ecological system lies in the fact that the ecological system with a greater species diversity  may function more efficiently and productively since more resources available for other species within the ecological system will be made. Therefore, the study of species diversity has been extensively investigated via both field research and theoretical approaches. 

As a suggestive example of species diversity, we investigate the situation where one exotic competing species (say, $W$) invades the ecological system of two native species (say, U and V) that are competing in the absence of $W$. Then a problem related to \textit{competitive exclusion} (\cite{Armstrong80Competitive-exclusion}, \cite{Hsu08Competitive-exclusion}, \cite{Hsu-Smith-Waltman96Competitive-exclusion-coexistence-Competitive}, \cite{Jang13Competitive-exclusion-Leslie-Gower-competition-Allee}, \cite{McGehee77Competitive-exclusion}, \cite{Smith94Competition}) or \textit{competitor-mediated coexistence} (\cite{CantrellWard97Competition-mediatedCoexistence}, \cite{Kastendiek82Competitor-mediatedCoexistence3Species}, \cite{Mimura15DynamicCoexistence3species}) arises, and a mathematical model for this situation can be provided by the diffusive Lotka-Volterra system of three competing species (\cite{Adamson12SpeciesCyclicCompetition}, \cite{EiMimuraIkota99SegregatingCompetition-diffusion}, \cite{Grossberg78Decision-Patterns-Oscillations-LVcompetitive}, \cite{Gyllenberg09LV3species-Heteroclinic}, \cite{Hallam79Persistence-Extinction3speciesLV}, \cite{KoRyuAhn14Coexistence3Competing-species}, \cite{Maier13Integration3-dimensionalLV}, \cite{Mimura15DynamicCoexistence3species}, \cite{PetrovskiiShigesada01Spatio-temporalThree-competitive-species}, \cite{Zeeman98Three-dimensionalCompetitiveLV}, \cite{Zeeman93Hopf-bifurcationsCompetitive3speciesLV})
\begin{equation}\label{eqn: compet L-V sys of 3 species with diffu}
\begin{cases}
\vspace{3mm}
u_t=d_1\,u_{yy}+u\,(\sigma_1-c_{11}\,u-c_{12}\,v-c_{13}\,w), \quad y\in\mathbb{R},\quad t>0,\\
\vspace{3mm}
v_t=d_2\,v_{yy}+v\,(\sigma_2-c_{21}\,u-c_{22}\,v-c_{23}\,w), \quad y\in\mathbb{R},\quad t>0,\\
w_t=d_3\,w_{yy}+w\,(\sigma_3-c_{31}\,u-c_{32}\,v-c_{33}\,w), \quad y\in\mathbb{R},\quad t>0,
\end{cases}
\end{equation} 
where $u(y,t)$, $v(y,t)$ and $w(y,t)$ denote the population densities of $U$, $V$ and $W$ at time $t$ and position $y$. The parameters $d_i$, $\sigma_i$, $c_{ii}$ $(i=1,2,3)$, and $c_{ij}$ $(i,j=1,2,3  \ \text{with} \;i\neq j)$, which are all positive constants, stand for the diffusion rates, intrinsic growth rates, intra-specific competition rates, and inter-specific competition rates, respectively. In particular, the result in \cite{CHMU-semi} indicates from the viewpoint of competitor-mediated coexistence, the coexistence of strongly competing species in the presence of an exotic competing species.

We begin with a two-species system of \eqref{eqn: compet L-V sys of 3 species with diffu} in the absence of $W$, i.e.
\begin{equation}\label{eqn: compet L-V sys of 2 species with diffu}
\begin{cases}
u_t=d_1\,u_{yy}+u\,(\sigma_1-c_{11}\,u-c_{12}\,v), \quad y\in\mathbb{R},\quad t>0,\\ \\
\hspace{0.7mm}v_t=d_2\,v_{yy}+v\,(\sigma_2-c_{21}\,u-c_{22}\,v), \quad y\in\mathbb{R},\quad t>0.\\
\end{cases}
\end{equation}
Imposing the zero Neumann boundary conditions
\begin{equation}
\frac{\partial u}{\partial \nu}=\frac{\partial v}{\partial \nu}=0,\quad x\in\partial\Omega
\end{equation}
and suitable initial conditions
\begin{equation}
u(x,0)=\hat{u}(x),\;v(x,0)=\hat{v}(x),\quad x\in\Omega
\end{equation}
on \eqref{eqn: compet L-V sys of 2 species with diffu} with the entire space $\mathbb{R}$ is replaced by a bounded and convex domain $\Omega$, we conclude from \cite{Hirsch83Strongly-monotone-semiflows} and \cite{KishimotoWeinberger85Stable-equilibriaConvex} that any positive solution $(u(x,t),v(x,t))$ of such a initial-boundary value problem converges to either $(\frac{\sigma_1}{c_{11}},0)$ or $(0,\frac{\sigma_2}{c_{22}})$ when $U$ and $V$ are \textit{strongly competing}, i.e. when the following condition hold:
\begin{equation}
\frac{\sigma_1}{c_{11}}>\frac{\sigma_2}{c_{21}},\,\frac{\sigma_2}{c_{22}}>\frac{\sigma_1}{c_{12}}.
\end{equation}
In this case, \textit{Gause's principle of competitive exclusion} occurs between the two species $U$ and $V$ when $U$ and $V$ competing for the same limited resources cannot stably coexist; one will prevail and the other is excluded. When the influence of diffusion in \eqref{eqn: compet L-V sys of 2 species with diffu} is disregarded, \eqref{eqn: compet L-V sys of 2 species with diffu} becomes
\begin{equation}\label{eqn: compet L-V sys of 2 species with diffu ODE}
\begin{cases}
u_t=u\,(\sigma_1-c_{11}\,u-c_{12}\,v),\quad t>0,\\ \\
\hspace{0.7mm}v_t=v\,(\sigma_2-c_{21}\,u-c_{22}\,v),\quad t>0.\\
\end{cases}
\end{equation}

It is readily seen that \eqref{eqn: compet L-V sys of 2 species with diffu ODE} has four equilibria: $\textbf{e}_1=(0,0)$, $\textbf{e}_2=\big(\frac{\sigma_1}{c_{11}},0\big)$, $\textbf{e}_3=\big(0,\frac{\sigma_2}{c_{22}}\big)$ and $\textbf{e}_4=(u^{\ast},v^{\ast})$, where $(u^{\ast},v^{\ast})=\big(\frac{c_{22}\,\sigma_1-c_{12}\,\sigma_2}{c_{11}\,c_{22}-c_{12}\,c_{21}},\frac{c_{11}\,\sigma_2-c_{21}\,\sigma_1}{c_{11}\,c_{22}-c_{12}\,c_{21}}\big)$ is the intersection of the two straight lines $\sigma_1-c_{11}\,u-c_{12}\,v=0$ and $\sigma_2-c_{21}\,u-c_{22}\,v=0$, whenever it exists. 
In the diffusion-free case, we can classify the asymptotic behaviour of solutions $(u(t),v(t))$ for \eqref{eqn: compet L-V sys of 2 species with diffu ODE} as $t\rightarrow\infty$ depending on $\sigma_i$ and $c_{ij}$ $(i,j=1,2)$, as described in:

\vspace{5mm}

\begin{prop}[\cite{demottoni79}]\label{prop: ODE stability of two-species systems}
Suppose that $(u(t),v(t))$ is a solution of \eqref{eqn: compet L-V sys of 2 species with diffu ODE} with initial conditions $u(0),v(0)>0$. We have 
\begin{itemize}
\item [(i)] \textit{(Competitive exclusion)} if   $\frac{\sigma_1}{c_{11}}>\frac{\sigma_2}{c_{21}}$ and $\frac{\sigma_2}{c_{22}}<\frac{\sigma_1}{c_{12}}$, then
                       $\lim\limits^{}_{t\rightarrow\infty}(u(t),v(t))=\big(\frac{\sigma_1}{c_{11}},0\big)$;
\item [(ii)] \textit{(Competitive exclusion)} if  $\frac{\sigma_1}{c_{11}}<\frac{\sigma_2}{c_{21}}$ and $\frac{\sigma_2}{c_{22}}>\frac{\sigma_1}{c_{12}}$, then
                       $\lim\limits^{}_{t\rightarrow\infty}(u(t),v(t))=\big(0,\frac{\sigma_2}{c_{22}}\big)$;
\item [(iii)] \textit{(Strong competition)}  if $\frac{\sigma_1}{c_{11}}>\frac{\sigma_2}{c_{21}}$ and $\frac{\sigma_2}{c_{22}}>\frac{\sigma_1}{c_{12}}$, then
$\lim\limits^{}_{t\rightarrow\infty}(u(t),v(t))=\big(\frac{\sigma_1}{c_{11}},0\big)$ or $\big(0,\frac{\sigma_2}{c_{22}}\big)$ depending on the initial condition;
\item [(iv)] \textit{(Weak competition)} if $\frac{\sigma_1}{c_{11}}<\frac{\sigma_2}{c_{21}}$ and $\frac{\sigma_2}{c_{22}}<\frac{\sigma_1}{c_{12}}$, then
                       $\lim\limits^{}_{t\rightarrow\infty}(u(t),v(t))=(u^{\ast},v^{\ast})$.
\end{itemize}
\end{prop}

\vspace{5mm}

As in case $(iii)$ mentioned previously, Gause's principle of competitive exclusion also occurs in cases $(i)$ and $(ii)$: one species always wins and the other species become extinct in the long run.   It is easy to see that we do not need to treat one of cases $(i)$ and $(ii)$ since one of the two cases is obtained from the other by exchanging the roles of $u$ and $v$. For the case of weak competition, the Lotka-Volterra model \eqref{eqn: compet L-V sys of 2 species with diffu ODE}  predicts that the stable coexistence state $(u^{\ast},v^{\ast})$ exists only when intra-specific competition has a greater effect than inter-specific competition.


We shall assume throughout, unless otherwise stated, that either strong or weak competition occurs between the two species $U$ and $V$:
\begin{itemize}
  \item \textbf{(Strong competition)} $\mathbf{[S]}$: $\frac{\sigma_1}{c_{11}}>\frac{\sigma_2}{c_{21}}$ and $\frac{\sigma_2}{c_{22}}>\frac{\sigma_1}{c_{12}}$; 
  \item  \textbf{(Weak competition)} $\mathbf{[W]}$: $\frac{\sigma_1}{c_{11}}<\frac{\sigma_2}{c_{21}}$ and $\frac{\sigma_2}{c_{22}}<\frac{\sigma_1}{c_{12}}$ 
\end{itemize}
in investigating \textit{traveling wave solutions} (\cite{Volperts94TWS-Parabolic})
\begin{equation}\label{eqn: traveling wave (u,v)=(U,V) 2 species}
(u(y,t),v(y,t))=(u(x),v(x)), \quad x=y-\theta \,t
\end{equation}
of \eqref{eqn: compet L-V sys of 2 species with diffu}. Here $\theta$ is the propagation speed of the traveling wave. We note that $u^{\ast},v^{\ast}>0$ if and only if either $\mathbf{[S]}$ or $\mathbf{[W]}$ holds. Substituting \eqref{eqn: traveling wave (u,v)=(U,V) 2 species} into \eqref{eqn: compet L-V sys of 2 species with diffu} yields the following system of ordinary differential equations
\begin{equation}\label{eqn: compet L-V sys of 2 species with diffu TWS}
\begin{cases}
d_1\,u_{xx}+\theta\,u_{x}+u\,(\sigma_1-c_{11}\,u-c_{12}\,v)=0, \quad x\in\mathbb{R}, \\\\
d_2\,v_{xx}\hspace{0.8mm}+\theta\,v_{x}+v\,(\sigma_2-c_{21}\,u-c_{22}\,v)=0,\quad x\in\mathbb{R}. \\
\end{cases}
\end{equation}
The problem as to which species will survive in a competitive system is of importance in ecology. In order to tackle this problem, we use traveling wave solutions of the form \eqref{eqn: traveling wave (u,v)=(U,V) 2 species}.
 
In this paper, we treat the following boundary value problem
\begin{equation}\label{eqn: L-V BVP  before scaled}
\begin{cases}
\vspace{3mm}
d_1\,u_{xx}+\theta\,u_{x}+u\,(\sigma_1-c_{11}\,u-c_{12}\,v)=0, \quad x\in\mathbb{R}, \\
\vspace{3mm}
d_2\,v_{xx}\hspace{0.8mm}+\theta\,v_{x}+v\,(\sigma_2-c_{21}\,u-c_{22}\,v)=0,\quad x\in\mathbb{R}, \\
(u,v)(-\infty)=\text{\bf e}_2,\quad (u,v)(+\infty)= \text{\bf e}_4.
\end{cases}
\end{equation}
and call a solution $(u(x),v(x))$ of \eqref{eqn: L-V BVP  before scaled} an $(\textbf{e}_2,\textbf{e}_4)$-wave.  


Under various parameters, monotone $(\textbf{e}_2,\textbf{e}_4)$-waves are investigated via different approaches (for instance, \cite{Hou&Leung08}, \cite{Kanel06}, \cite{Kanel&Zhou96estimatespeed}). It is not clear from the assumptions of parameters given in the references above that the monotone $(\textbf{e}_2,\textbf{e}_4)$-waves have the property $u^{\ast}> v^{\ast}$ or not. Let us see what happens when the answer is affirmative. If $u^{\ast}> v^{\ast}$ holds, then we easily see that, since $v$ is monotonically increasing and $u$ is monotonically decreasing the profile of $u$ lies completely above the profile of $v$. Accordingly, in this case $u$ dominates the entire habitat $\mathbb{R}$. Indeed, it has been proved that there exist two types of $(\textbf{e}_2,\textbf{e}_4)$-waves, one with $u^{\ast}> v^{\ast}$ and the other with $u^{\ast}\leq v^{\ast}$, by giving exact $(\textbf{e}_2,\textbf{e}_4)$-waves in \cite{hung2012JJIAM}. We note that in the case $u^{\ast}> v^{\ast}$, the phenomenon exhibited by the dominance of $u$ in the entire habitat $\mathbb{R}$ is unique and is yet to be explored. In particular, exact $(\textbf{e}_1,\textbf{e}_2)$-waves, $(\textbf{e}_1,\textbf{e}_4)$-waves, $(\textbf{e}_2,\textbf{e}_3)$-waves and $(\textbf{e}_2,\textbf{e}_4)$-waves are given in \cite{hung2012JJIAM} under certain conditions on the parameters by applying judicious ans\"{a}tzes for exact solutions. 

When the inhabitant of the two competing species $U$ and $V$ is resource-limited, the investigation of the total mass or the total density of the two species $U$ and $V$ is essential. This gives rises to the problem as to the estimate of the total density $u(x)+v(x)$ in \eqref{eqn: L-V BVP  before scaled}. In addition, another issue which motivates us to study the estimate of $u(x)+v(x)$ is the measurement of the \textit{species evenness index} $\mathcal{J}$ for \eqref{eqn: L-V BVP  before scaled}. As $\mathcal{J}$ is defined via the \textit{Shannon's diversity index} $\mathcal{H}$ (\cite{Baczkowski98Generalized-diversity-index}, \cite{Good53Estimation-population-parameters}, \cite{Ramezan11Shannon-diversity-index},  \cite{Simpson49Measurement-diversity}), i.e.
\begin{equation}
\mathcal{J}=\frac{\mathcal{H}}{\ln (s)},
\end{equation}
where
\begin{equation}
\mathcal{H}=-\sum_{i=1}^{s} p_i\cdot\ln (p_i),
\end{equation}
$s$ is the total number of species, and $p_i$ is the proportion of the $i$-th species determined by dividing the number of the $i$-th species species by the total number of all species, the species evenness index $\mathcal{J}$ for \eqref{eqn: L-V BVP  before scaled} is given by
\begin{equation}\label{eqn: species evenness index J for two species}
\mathcal{J}=-\frac{1}{(\ln 2) (u+v)}\,\left(u\,\ln \Big(\frac{u}{u+v}\Big)+v\,\ln \Big(\frac{v}{u+v}\Big)\right).
\end{equation}
One of our primary goals in this paper is to address the problem of giving a priori estimates of $u(x)+v(x)$, which is involved in the calculation of $\mathcal{J}$ in \eqref{eqn: species evenness index J for two species}. On the other hand, we also are concerned with the following question when a priori estimates of $u(x)+v(x)$ are given:

\vspace{5mm}

\textbf{Q1}: \textit{How does the estimate of $u(x)+v(x)$ depend on the parameters in \eqref{eqn: L-V BVP  before scaled}?}

\vspace{5mm}

In \cite{Chen&Hung15Nonexistence}, upper and lower bounds of $u(x)+v(x)$ are given when the two diffusion rates $d_1$ and $d_2$ are equal. However, the approach employed in \cite{Chen&Hung15Nonexistence} to obtain upper or lower bounds for $u+v$ cannot be applied directly to the case where the diffusion rates $d_1$ and $d_2$ are not equal. 

\vspace{5mm}

\textbf{Q2}: \textit{In \eqref{eqn: L-V BVP  before scaled}, when $d_1\neq d_2$, can upper and lower bounds of $u+v$ be obtained?}

\vspace{5mm}

As for the answer to \textbf{Q2}, it seems as far as we know, not available in the literature. To give an affirmative answer to this question, we develop a new but elementary approach. Moreover, employing this approach leads to a affirmative answer to the following question which is more general than \textbf{Q2}:

\vspace{5mm}

\textbf{Q3}: \textit{In \eqref{eqn: L-V BVP  before scaled}, when $d_1\neq d_2$, can upper and lower bounds of $\alpha\,u+\beta\,v$, where $\alpha,\beta>0$ are arbitrary constants, be given?}

\vspace{5mm}



By adding the two equations in \eqref{eqn: L-V BVP  before scaled}, we obtain an equation involving $p(x)=\alpha\,u(x)+\beta\,v(x)$ and $q(x)=d_1\,\alpha\,u(x)+d_2\,\beta\,v(x)$, i.e. 
\begin{align}\label{eq: q''+p'+f+g=0 before scale}
  0&=\alpha\,\big(d_1\,u_{xx}+\theta\,u_{x}+u\,(\sigma_1-c_{11}\,u-c_{12}\,v)\big)+
      \beta\,\big(d_2\,v_{xx}+\theta\,v_{x}+v\,(\sigma_2-c_{21}\,u-c_{22}\,v)\big)\notag\\[3mm]
  &=q''(x)+\theta\,p'(x)+ 
  \alpha\,u\,(\sigma_1-c_{11}\,u-c_{12}\,v)+
    \beta\,v\,(\sigma_2-c_{21}\,u-c_{22}\,v)\notag\\[3mm]
  &=q''(x)+\theta\,p'(x)+F(u,v).
\end{align}
For the sake of clear exposition we shall assume from now on that $F(u,v)=  \alpha\,u\,(\sigma_1-c_{11}\,u-c_{12}\,v)+\beta\,v\,(\sigma_2-c_{21}\,u-c_{22}\,v)$.
The case where $d_1=d_2$ or $p(x)$ is a non-zero constant multiple of $q(x)$ has been considered in \cite{Chen&Hung15Nonexistence}. A mathematical difficulty arises as a consequence of the fact that the approach used in \cite{Chen&Hung15Nonexistence} cannot be directly applied, when $d_1\neq d_2$ in the last equation since $p(x)$ no long can be written as a constant multiple of $q(x)$ and such $p(x)$ and $q(x)$ are involved in a \textit{single equation} like \eqref{eq: q''+p'+f+g=0 before scale}. The approach proposed here can be employed to give estimates of $p(x)$ in the case where $p(x)$ and $q(x)$ are involved in the single equation \eqref{eq: q''+p'+f+g=0 before scale}. One of the main results of this paper is the N-barrier maximum principle (Theorem~\ref{thm: lb<q(x)<ub}), which gives an affirmative answer to \textbf{Q3}.

The rest of the paper is organized as follows. In the next section, the main results of this paper, including the N-barrier maximum principle (Theorem~\ref{thm: lb<q(x)<ub}) and two applications (Theorem~\ref{thm: Existence 3 species} and Theorem~\ref{thm: Nonexistence 3 species}) to the system of three species \eqref{eqn: compet L-V sys of 3 species with diffu}, are summarized. We prove in Section~\ref{sec: main result} the N-barrier maximum principle by means of the construction of N-barriers depending on various conditions. Under certain restrictions on the parameters, the existence of exact traveling waves solutions for \eqref{eqn: compet L-V sys of 3 species with diffu} (Theorem~\ref{thm: exact (1,0,0)-(u,v,0) wave}) is presented in Section~\ref{sec: exact (1,0,0)-(u,v,0) waves}. Finally, Section~\ref{sec: Applications of N-barrier maxi principle } is devoted to the proofs of Theorem~\ref{thm: Existence 3 species} and Theorem~\ref{thm: Nonexistence 3 species}.  





\vspace{5mm}
\setcounter{equation}{0}
\setcounter{figure}{0}
\setcounter{subfigure}{0}
\section{Statement of main theorems}\label{sec: statement of theorems}
\vspace{5mm}

\begin{thm}[\textbf{N-barrier maximum principle}]\label{thm: lb<q(x)<ub}
Under either $\mathbf{[S]}$ or $\mathbf{[W]}$, we assume that $(u(x),v(x))$ with $u(x), v(x)>0$ for $x\in\mathbb{R}$ satisfies the boundary value problem \eqref{eqn: L-V BVP  before scaled}.
For any $\alpha, \beta>0$, we have
\begin{equation}
q_{\ast}
\leq \alpha\,u(x)+\beta\,v(x)\leq 
q^{\ast}, \quad x\in\mathbb{R},
\end{equation}
where 
\begin{equation}\label{eqn: lower bed of q in statement of main theorems}
q_{\ast}=\min\bigg[\alpha\,\min\Big[\frac{\sigma_1}{c_{11}},\frac{\sigma_2}{c_{21}}\Big],\beta\,\min\Big[\frac{\sigma_2}{c_{22}},\frac{\sigma_1}{c_{12}}\Big]\bigg]\min\Big[\frac{d_1}{d_2},\frac{d_2}{d_1}\Big]
\end{equation}
and
\begin{equation}\label{eqn: upper bed of q in statement of main theorems}
q^{\ast}=\max\bigg[\alpha\,\max\Big[\frac{\sigma_1}{c_{11}},\frac{\sigma_2}{c_{21}}\Big],\beta\,\max\Big[\frac{\sigma_2}{c_{22}},\frac{\sigma_1}{c_{12}}\Big]\bigg]\max\Big[\frac{d_1}{d_2},\frac{d_2}{d_1}\Big].
\end{equation}
In particular,
\begin{itemize}
  \item [(I)] (\textbf{Lower bounds for $q=q(x)$})  when the differential equations in \eqref{eqn: L-V BVP  before scaled} are replaced by the differential inequalities
\begin{equation}\label{eqn: L-V <0}
\begin{cases}
\vspace{3mm}
d_1\,u_{xx}+\theta\,u_{x}+u\,(\sigma_1-c_{11}\,u-c_{12}\,v)\leq0, \quad x\in\mathbb{R}, \\
d_2\,v_{xx}\hspace{0.8mm}+\theta\,v_{x}+v\,(\sigma_2-c_{21}\,u-c_{22}\,v)\leq0,\quad x\in\mathbb{R}, \\
\end{cases}
\end{equation}  
we have $\alpha\,u(x)+\beta\,v(x)\geq q_{\ast}$ for $x\in\mathbb{R}$;
  \item [(II)] (\textbf{Upper bounds for $q=q(x)$}) when the differential equations in \eqref{eqn: L-V BVP  before scaled} are replaced by the differential inequalities
\begin{equation}\label{eqn: L-V >0}
\begin{cases}
\vspace{3mm}
d_1\,u_{xx}+\theta\,u_{x}+u\,(\sigma_1-c_{11}\,u-c_{12}\,v)\geq0, \quad x\in\mathbb{R}, \\
d_2\,v_{xx}\hspace{0.8mm}+\theta\,v_{x}+v\,(\sigma_2-c_{21}\,u-c_{22}\,v)\geq0,\quad x\in\mathbb{R}, \\
\end{cases}
\end{equation}  
we have $\alpha\,u(x)+\beta\,v(x)\leq q^{\ast}$ for $x\in\mathbb{R}$.
\end{itemize}
\end{thm}

\vspace{5mm}
We would like to add a few remarks concerning Theorem~\ref{thm: lb<q(x)<ub}. 

\begin{itemize}
  \item [(i)] In addition to the diffusion rates $d_1$ and $d_2$, upper and lower bounds $q^{\ast}$ and $q_{\ast}$ depend only on the $u$-intersection (i.e. $\frac{\sigma_1}{c_{11}}$) of the line $\sigma_1-c_{11}\,u-c_{12}\,v=0$ and $v$-intersection (i.e. $\frac{\sigma_2}{c_{22}}$) of the line $\sigma_2-c_{21}\,u-c_{22}\,v=0$ as well as the $v$-intersection (i.e. $\frac{\sigma_1}{c_{12}}$) of the line $\sigma_1-c_{11}\,u-c_{12}\,v=0$ and $u$-intersection (i.e. $\frac{\sigma_2}{c_{21}}$) of the line $\sigma_2-c_{21}\,u-c_{22}\,v=0$, respectively. We note that $\textbf{e}_2=(\frac{\sigma_1}{c_{11}},0)$ and $\textbf{e}_3=(0,\frac{\sigma_2}{c_{22}})$ represent the two \textit{competitively exclusive states}. When $\alpha=\beta=1$, the above observation together with the fact that the estimate of $\alpha\,u(x)+\beta\,v(x)$ given in Theorem~\ref{thm: lb<q(x)<ub} does not explicitly depend on the propagating speed $\theta$ gives a possible answer to \textbf{Q1}. 
  

  \item [(ii)]  Letting $\alpha=1$ and $\beta\rightarrow0$ in Theorem~\ref{thm: lb<q(x)<ub}, we obtain
  \begin{equation}
0\leq u(x)\leq \max\Big[\frac{\sigma_1}{c_{11}},\frac{\sigma_2}{c_{21}}\Big] \max\Big[\frac{d_1}{d_2},\frac{d_2}{d_1}\Big].
\end{equation}
Letting $\beta=1$ and $\alpha\rightarrow0$ in Theorem~\ref{thm: lb<q(x)<ub}, we obtain
  \begin{equation}
0\leq v(x)\leq \max\Big[\frac{\sigma_2}{c_{22}},\frac{\sigma_1}{c_{12}}\Big] \max\Big[\frac{d_1}{d_2},\frac{d_2}{d_1}\Big].
\end{equation}
\item [(iii)] 
 When $\mathbf{[S]}$: $\frac{\sigma_1}{c_{11}}>\frac{\sigma_2}{c_{21}}$ and $\frac{\sigma_2}{c_{22}}>\frac{\sigma_1}{c_{12}}$ holds, $q^{\ast}$ and $q_{\ast}$ are given by 
 \begin{equation}
q_{\ast}=\min\bigg[\alpha\,\frac{\sigma_2}{c_{21}},\beta\,\frac{\sigma_1}{c_{12}}\bigg]\min\Big[\frac{d_1}{d_2},\frac{d_2}{d_1}\Big]
\end{equation}
and
\begin{equation}
q^{\ast}=\max\bigg[\alpha\,\frac{\sigma_1}{c_{11}},\beta\,\frac{\sigma_2}{c_{22}}\bigg]\max\Big[\frac{d_1}{d_2},\frac{d_2}{d_1}\Big].
\end{equation}
When $\mathbf{[W]}$: $\frac{\sigma_1}{c_{11}}<\frac{\sigma_2}{c_{21}}$ and $\frac{\sigma_2}{c_{22}}<\frac{\sigma_1}{c_{12}}$ holds, $q^{\ast}$ and $q_{\ast}$ are given by 
 \begin{equation}
q_{\ast}=\min\bigg[\alpha\,\frac{\sigma_1}{c_{11}},\beta\,\frac{\sigma_2}{c_{22}}\bigg]\min\Big[\frac{d_1}{d_2},\frac{d_2}{d_1}\Big]
\end{equation}
and
\begin{equation}
q^{\ast}=\max\bigg[\alpha\,\frac{\sigma_2}{c_{21}},\beta\,\frac{\sigma_1}{c_{12}}\bigg]\max\Big[\frac{d_1}{d_2},\frac{d_2}{d_1}\Big].
\end{equation}

\item [(iv)] By letting $\alpha=\beta=1$, we have
\begin{equation}
q_{\ast}
\leq u(x)+v(x)\leq 
q^{\ast}, \quad x\in\mathbb{R},
\end{equation}
where 
\begin{equation}
q_{\ast}=\min\bigg[\min\Big[\frac{\sigma_1}{c_{11}},\frac{\sigma_2}{c_{21}}\Big],\min\Big[\frac{\sigma_2}{c_{22}},\frac{\sigma_1}{c_{12}}\Big]\bigg]\min\Big[\frac{d_1}{d_2},\frac{d_2}{d_1}\Big]
\end{equation}
and
\begin{equation}
q^{\ast}=\max\bigg[\max\Big[\frac{\sigma_1}{c_{11}},\frac{\sigma_2}{c_{21}}\Big],\max\Big[\frac{\sigma_2}{c_{22}},\frac{\sigma_1}{c_{12}}\Big]\bigg]\max\Big[\frac{d_1}{d_2},\frac{d_2}{d_1}\Big].
\end{equation}
This answers \textbf{Q2}.
\end{itemize}

\vspace{5mm}


Let us return to the system of three species \eqref{eqn: compet L-V sys of 3 species with diffu} and consider traveling wave solutions
\begin{equation}\label{eqn: traveling wave (u,v)=(U,V)}
(u(y,t),v(y,t),w(y,t))=(u(x),v(x),w(x)), \quad x=y-\theta \,t
\end{equation}
satisfying the following boundary value problem
\begin{equation}\label{eqn: L-V systems of three species (TWS)}
\begin{cases}
\vspace{3mm}
d_1\,u_{xx}\hspace{-0.5mm}+\theta \,u_x\hspace{-0.5mm}+u\,(\sigma_1-c_{11}\,u-c_{12}\,v-c_{13}\,w)=0, \quad x\in\mathbb{R}, \\
\vspace{3mm}
d_2\,v_{xx}+\theta \,v_x+v\,(\sigma_2-c_{21}\,u-c_{22}\,v-c_{23}\,w)=0, \quad x\in\mathbb{R},\\
\vspace{3mm}
d_3\,w_{xx}\hspace{-1mm}+\theta \,w_x\hspace{-1.5mm}+\hspace{-0.5mm}w\,(\sigma_3-c_{31}\,u-c_{32}\,v-c_{33}\,w)=0, \quad x\in\mathbb{R},\\
(u,v,w)(-\infty)=(\frac{\sigma_1}{c_{11}},0,0), \quad (u,v,w)(\infty)=(u^{\ast},v^{\ast},0).
\end{cases}
\end{equation}
Here again $\theta$ is the propagation speed of the traveling wave. As mentioned in the beginning of introduction, we will study the influence of an exotic species $W$ on other native species $U$ and $V$ in terms of \eqref{eqn: L-V systems of three species (TWS)}. The first question we shall ask is whether competitor-mediated coexistence occurs for $u$, $v$, and $w$ in the system \eqref{eqn: L-V systems of three species (TWS)}. If the three species do coexist under certain conditions, then what will be the profiles of $u(x)$, $v(x)$, and $w(x)$? The result in \cite{hung2012JJIAM} indicates that when $w(x)$ is absent in \eqref{eqn: L-V systems of three species (TWS)}, the system of two species \eqref{eqn: L-V BVP  before scaled} under certain conditions admits solutions $(u(x),v(x))$ having the profiles with $u(x)$ being monotonically decreasing and $v(x)$ being monotonically increasing. Moreover, we see from the profiles of $u(x)$ and $v(x)$ that $u(x)$ and $v(x)$ dominate the neighborhood of $x=-\infty$ and the neighborhood of $x=\infty$, respectively. These facts lead us to the expectation that, the profile of $w(x)$ must be \textit{pulse-like} ($w(x)$ is a pulse if $w(-\infty)=w(\infty)=0$ and $w(x)>0$ for $x\in\mathbb{R}$) if it exists since $w$ will prevail only when $u$ and $v$ are not dominant.



To simplify the problem, we restrict ourself to the case of $\sigma_1=c_{11}$ in Section~\ref{sec: exact (1,0,0)-(u,v,0) waves} and denote a solution of \eqref{eqn: L-V systems of three species (TWS)} with $\sigma_1=c_{11}$ by $(1,0,0)$-$(u^{\ast},v^{\ast},0)$ wave for convenience. Although we can find exact $(1,0,0)$-$(u^{\ast},v^{\ast},0)$ waves for \eqref{eqn: L-V systems of three species (TWS)} (See Theorem~\ref{thm: exact (1,0,0)-(u,v,0) wave} and we remark that when $\sigma_1\neq c_{11}$, a similar result remains valid.) under certain restrictions on the parameters, it remains an open problem, however, to establish the existence of solutions for \eqref{eqn: L-V systems of three species (TWS)} under more general conditions. In spite of this fact, when we consider the situation where the influence of the invading species $W$ on the native species $U$ and $V$ is of no significance, i.e. $c_{13}, c_{23}\approx0$ in \eqref{eqn: L-V systems of three species (TWS)}, the limiting case $c_{13}, c_{23}\rightarrow0^{+}$ leads to the boundary value problem
\begin{equation}\label{eqn: L-V systems of three species (TWS) c13=c23=0}
\begin{cases}
\vspace{3mm}
d_1\,u_{xx}\hspace{-0.5mm}+\theta \,u_x\hspace{-0.5mm}+u\,(\sigma_1-c_{11}\,u-c_{12}\,v)=0, \quad x\in\mathbb{R}, \\
\vspace{3mm}
d_2\,v_{xx}+\theta \,v_x+v\,(\sigma_2-c_{21}\,u-c_{22}\,v)=0, \quad x\in\mathbb{R},\\
\vspace{3mm}
d_3\,w_{xx}\hspace{-1mm}+\theta \,w_x\hspace{-1.5mm}+\hspace{-0.5mm}w\,(\sigma_3-c_{31}\,u-c_{32}\,v-c_{33}\,w)=0, \quad x\in\mathbb{R},\\
(u,v,w)(-\infty)=(\frac{\sigma_1}{c_{11}},0,0), \quad (u,v,w)(\infty)=(u^{\ast},v^{\ast},0).
\end{cases}
\end{equation}
Under the assumption of the existence of solutions $(u(x),v(x))=(\tilde{u}(x),\tilde{v}(x))$ for the system of two species \eqref{eqn: L-V BVP  before scaled} (\cite{Hou&Leung08}, \cite{hung2012JJIAM}, \cite{Kanel06}, \cite{Kanel&Zhou96estimatespeed}), it will be proved in Section~\ref{subsec: Application to the existence} that under certain conditions, a solution $w(x)$ of the third equation in \eqref{eqn: L-V systems of three species (TWS) c13=c23=0}, i.e. the non-autonomous Fisher equation for $w=w(x)$
\begin{equation}\label{eqn: eqn of Fisher type statement of main theorems}
d_3\,w_{xx}+\theta \,w_x+w\,(\sigma_3-c_{31}\,\tilde{u}-c_{32}\,\tilde{v}-c_{33}\,w)=0, \quad x\in \mathbb{R}
\end{equation}
can be found applying the \textit{supersolution-subsolution method}, thereby establishing the existence of solutions for \eqref{eqn: L-V systems of three species (TWS) c13=c23=0}. We remark that, as an application of the N-barrier maximum principle, upper and lower bounds of $c_{31}\,\tilde{u}+c_{32}\,\tilde{v}$ are used in constructing supersolutions and subsolutions of \eqref{eqn: eqn of Fisher type statement of main theorems}.

\vspace{5mm}

\begin{thm}[\textbf{Existence of traveling wave solutions for three competing species}]
\label{thm: Existence 3 species}
Assume either $\mathbf{[S]}$ or $\mathbf{[W]}$. Suppose that there exist $u=\tilde{u}(x)$ and $v=\tilde{v}(x)$ which solve the first two equations in \eqref{eqn: L-V systems of three species (TWS) c13=c23=0} for some $\theta$, $d_i$, $\sigma_i$, $c_{ii}$ $(i=1,2)$, and $c_{ij}$ $(i,j=1,2  \ \text{with} \;i\neq j)$ and satisfy the boundary conditions $(u,v)(-\infty)=(1,0),(u,v)(\infty)=(u^{\ast},v^{\ast})$. Let 
\begin{equation}
\underaccent\bar{q}=\min\bigg[c_{31}\min\Big[\frac{\sigma_1}{c_{11}},\frac{\sigma_2}{c_{21}}\Big],c_{32}\min\Big[\frac{\sigma_2}{c_{22}},\frac{\sigma_1}{c_{12}}\Big]\bigg]\min\Big[\frac{d_1}{d_2},\frac{d_2}{d_1}\Big]
\end{equation}
and
\begin{equation}
\bar{q}=\max\bigg[c_{31}\max\Big[\frac{\sigma_1}{c_{11}},\frac{\sigma_2}{c_{21}}\Big],c_{32}\max\Big[\frac{\sigma_2}{c_{22}},\frac{\sigma_1}{c_{12}}\Big]\bigg]\max\Big[\frac{d_1}{d_2},\frac{d_2}{d_1}\Big].
\end{equation}
Assume that the following hypotheses hold:
\begin{itemize}
\item [$\mathbf{[H1]}$] $\sigma_3<c_{31}$, $\sigma_3<c_{31}\,u^{\ast}+c_{32}\,v^{\ast}$; 
\item [$\mathbf{[H2]}$] $-c_{33}\, \bar{K}-\underaccent\bar{q}+\sigma _3\leq0$;
\item [$\mathbf{[H3]}$] $4 \,\theta ^2-4 \left(c_{33}\,\underaccent\bar{K}+6\,d_3\right)\left(-c_{33}\, \underaccent\bar{K}-2\,d_3-\bar{q}+\sigma _3\right)\leq0$; 
\item [$\mathbf{[H4]}$] $\bar{K}\geq \underaccent\bar{K}>0$;
\end{itemize}
Then \eqref{eqn: L-V systems of three species (TWS) c13=c23=0} has a positive solution $(\tilde{u}(x),\tilde{v}(x),w(x))$ with $w(x)\not\equiv0$ for $x\in\mathbb{R}$ and $w(x)\rightarrow0$ as $x\rightarrow\pm\infty$. Moreover, $\underaccent\bar{w}(x)\leq w(x)\leq \bar{w}(x)$ for $x\in\mathbb{R}$, where $\underaccent\bar{w}(x)=\underaccent\bar{K}\,\big[\,1-\tanh^2(x)\,\big]$ and $\bar{w}(x)=\bar{K}$. 

\end{thm}

\vspace{5mm}
Another consequence of the N-barrier maximum principle concerns the nonexistence of solutions for \eqref{eqn: L-V systems of three species (TWS)}. In other words, we look for conditions on the parameters under which there exists no positive solution $(u(x),v(x),w(x))$ for \eqref{eqn: L-V systems of three species (TWS)}. 

\vspace{5mm}

\begin{thm}[\textbf{Nonexistence of traveling wave solutions for three competing species}]\label{thm: Nonexistence 3 species}
Let $\Sigma_1=\sigma_1\,c_{33}-\sigma_3\,c_{13}$ and $\Sigma_2=\sigma_2\,c_{33}-\sigma_3\,c_{23}$. Assume that the following hypotheses hold:
\begin{itemize}
\item [$\mathbf{[A1]}$] $\Sigma_1,\Sigma_2>0$; 
\item [$\mathbf{[A2]}$] either $c_{21}\,\Sigma_1>c_{11}\,\Sigma_2, c_{12}\,\Sigma_2>c_{22}\,\Sigma_1$ or $c_{21}\,\Sigma_1<c_{11}\,\Sigma_2, c_{12}\,\Sigma_2<c_{22}\,\Sigma_1$;
\item [$\mathbf{[A3]}$] 
$\min\bigg[c_{31}\,d_1\min\Big[\frac{\displaystyle\Sigma_1}{\displaystyle c_{11}},\frac{\displaystyle\Sigma_2}{\displaystyle c_{21}}\Big],c_{32}\,d_2\min\Big[\frac{\displaystyle\Sigma_2}{\displaystyle c_{22}},\frac{\displaystyle\Sigma_1}{\displaystyle c_{12}}\Big]\bigg]\min\Big[\frac{\displaystyle d_1}{\displaystyle d_2},\frac{\displaystyle d_2}{\displaystyle d_1}\Big]\geq\sigma_3\,c_{33}$.
\end{itemize}
Then \eqref{eqn: L-V systems of three species (TWS)} has no positive solution $(u(x),v(x),w(x))$.
\end{thm}


\vspace{5mm}
It will be clear from the proof in Section~\ref{sec: main result} that $\mathbf{[A1]}$ and $\mathbf{[A2]}$ assure that the N-barrier maximum principle can be applied in proving Theorem~\ref{thm: Nonexistence 3 species}. We note in particular that when $\sigma_3$ is sufficiently small, $\mathbf{[A3]}$ in Theorem~\ref{thm: Nonexistence 3 species} clearly holds. From the viewpoint of ecology, the result of Theorem~\ref{thm: Nonexistence 3 species} states that as the birth rate $\sigma_3$ of the species $W$ decreases below a threshold, the three species $U$, $V$ and $W$ no longer coexist. Intuitively, due to the weakness of the exotic species $W$, competitor-mediated coexistence cannot occur for the three species $U$, $V$ and $W$ in \eqref{eqn: L-V systems of three species (TWS)}.

\vspace{5mm}
\setcounter{equation}{0}
\setcounter{figure}{0}
\setcounter{subfigure}{0}
\section{N-barrier maximum principle: proof of Theorem~\ref{thm: lb<q(x)<ub}}\label{sec: main result}
\vspace{5mm}


In this section, we use the notations $p(x)=\alpha\,u(x)+\beta\,v(x)$, $q(x)=d_1\,\alpha\,u(x)+d_2\,\beta\,v(x)$, and $F(u,v)=  \alpha\,u\,(\sigma_1-c_{11}\,u-c_{12}\,v)+\beta\,v\,(\sigma_2-c_{21}\,u-c_{22}\,v)$ as in \eqref{eq: q''+p'+f+g=0 before scale}. We begin with a useful lemma. 


\vspace{5mm}

\begin{lem}\label{lem: quadratic curve of F(u,v)=0}
For the quadratic curve $F(u,v)=0$ in the $uv$-plane,
\begin{itemize}
\item [(i)] if $\frac{\sigma_1}{c_{11}}>\frac{\sigma_2}{c_{21}}$ and $\frac{\sigma_2}{c_{22}}>\frac{\sigma_1}{c_{12}}$, then $F(u,v)=0$ represents a hyperbola;
\item [(ii)] if $\frac{\sigma_1}{c_{11}}<\frac{\sigma_2}{c_{21}}$ and $\frac{\sigma_2}{c_{22}}<\frac{\sigma_1}{c_{12}}$, then $F(u,v)=0$ represents a hyperbola, a parabola, or an ellipse depending on the parameters in $F(u,v)=0$. 
\end{itemize}
\end{lem}
\begin{proof}
To prove $(i)$, we calculate the discriminant $\mathcal{D}$ of the quadratic curve $F(u,v)=\alpha\,u\,(1-u-a_1\,v)+\beta\,k\,v\,(1-a_2\,u-v)=0$ to obtain
\begin{equation}
\mathcal{D}=\left(\alpha\,c_{12}+\beta\,c_{21}\right)^2-4 \,\alpha\,\beta\, c_{11}\,c_{22}.
\end{equation}
Because of the assumption $\frac{\sigma_1}{c_{11}}>\frac{\sigma_2}{c_{21}}$ and $\frac{\sigma_2}{c_{22}}>\frac{\sigma_1}{c_{12}}$, it follows that $\left(\alpha\,c_{12}+\beta\,c_{21}\right)^2\geq 4\,\alpha\,\beta\,c_{12}\,c_{21}>4\,\alpha\,\beta\,(c_{22}\,\sigma_1\,\sigma_2^{-1})\,(c_{11}\,\sigma_1^{-1}\,\sigma_2)=4 \,\alpha\,\beta\, c_{11}\,c_{22}$. Thus $\mathcal{D}>0$ and the quadratic curve $F(u,v)=0$ is a hyperbola under the assumption $\frac{\sigma_1}{c_{11}}>\frac{\sigma_2}{c_{21}}$ and $\frac{\sigma_2}{c_{22}}>\frac{\sigma_1}{c_{12}}$.

For simplicity we let $\sigma_1=\sigma_2=c_{11}=c_{22}=1$, $c_{12}=\frac{1}{2}$, and $c_{21}=\frac{2}{3}$ to show $(ii)$. Depending on the other parameters, it is shown in Fig~\ref{fig: 6 figures of F(u,v)=0} that $F(u,v)=0$ represents a hyperbola, a parabola, or an ellipse in the $uv$-plane.

\begin{figure}[ht!]
\centering
\mbox{
\subfigure[]{\includegraphics[width=0.395\textwidth]{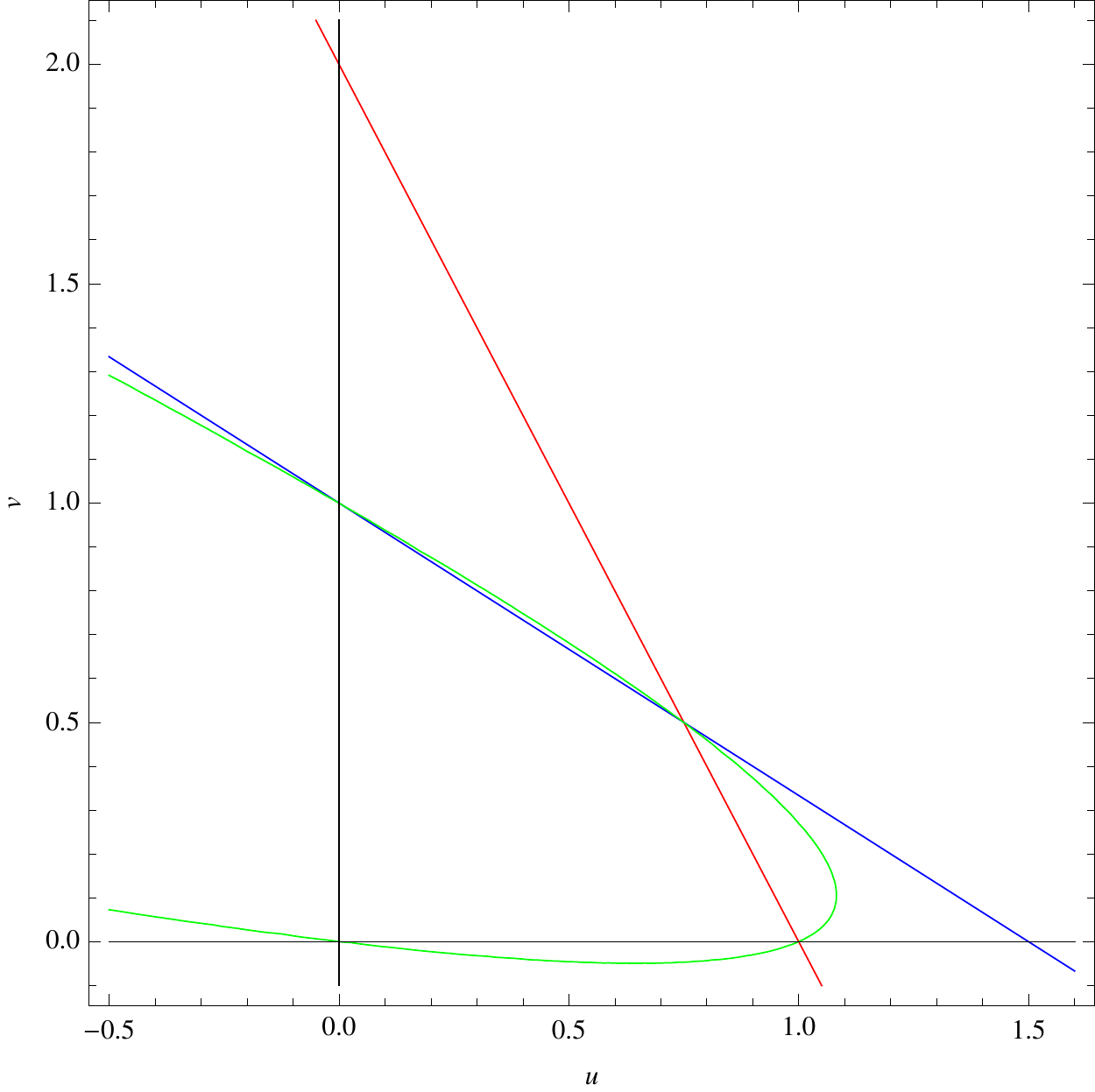}
             \label{fig: hyperbola_3_S}    } \quad \hspace{0mm}
\subfigure[]{\includegraphics[width=0.395\textwidth]{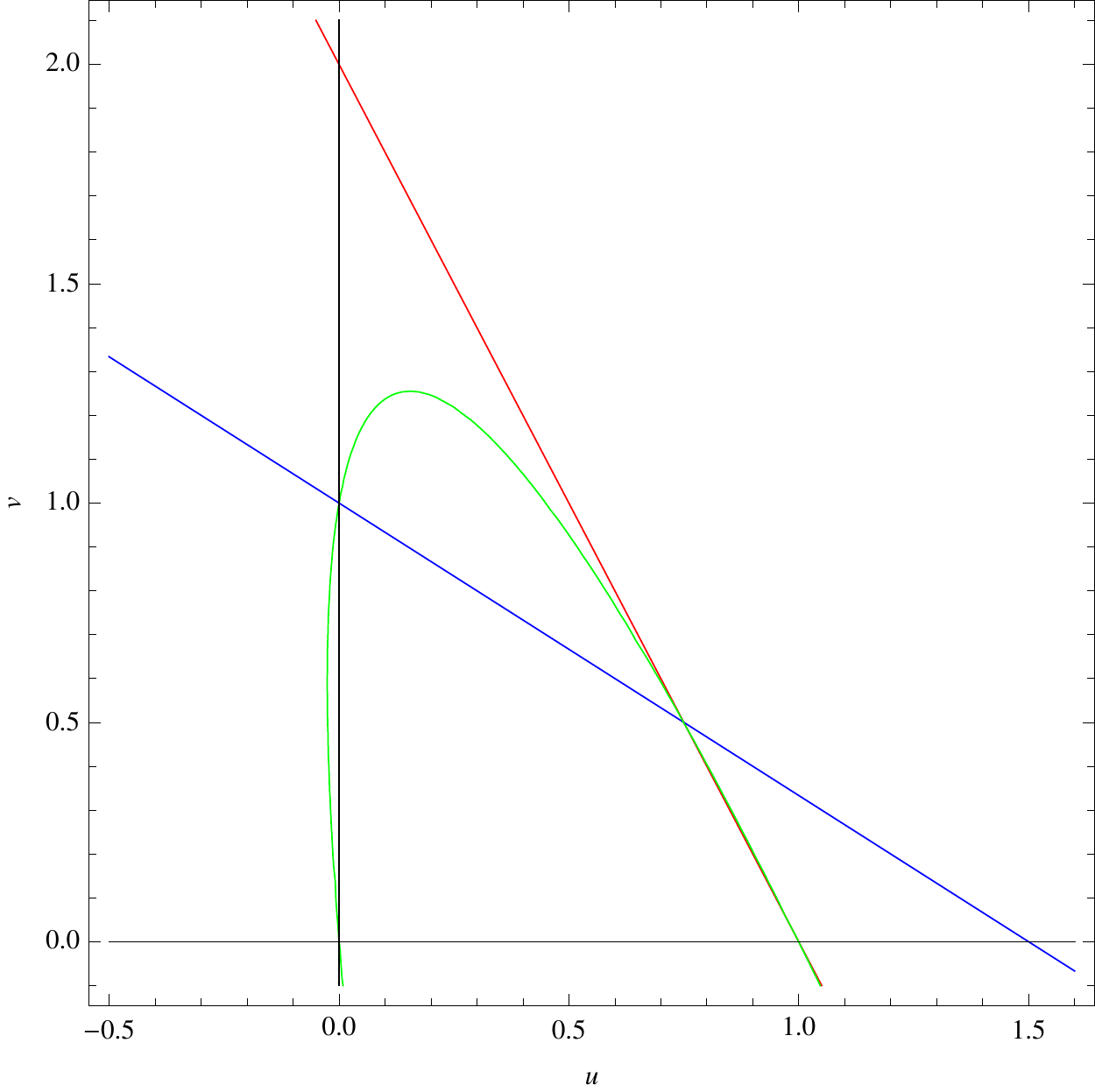}
             \label{fig: hyperbola_2}    } \quad \hspace{0mm}
             }
\mbox{
\subfigure[]{\includegraphics[width=0.395\textwidth]{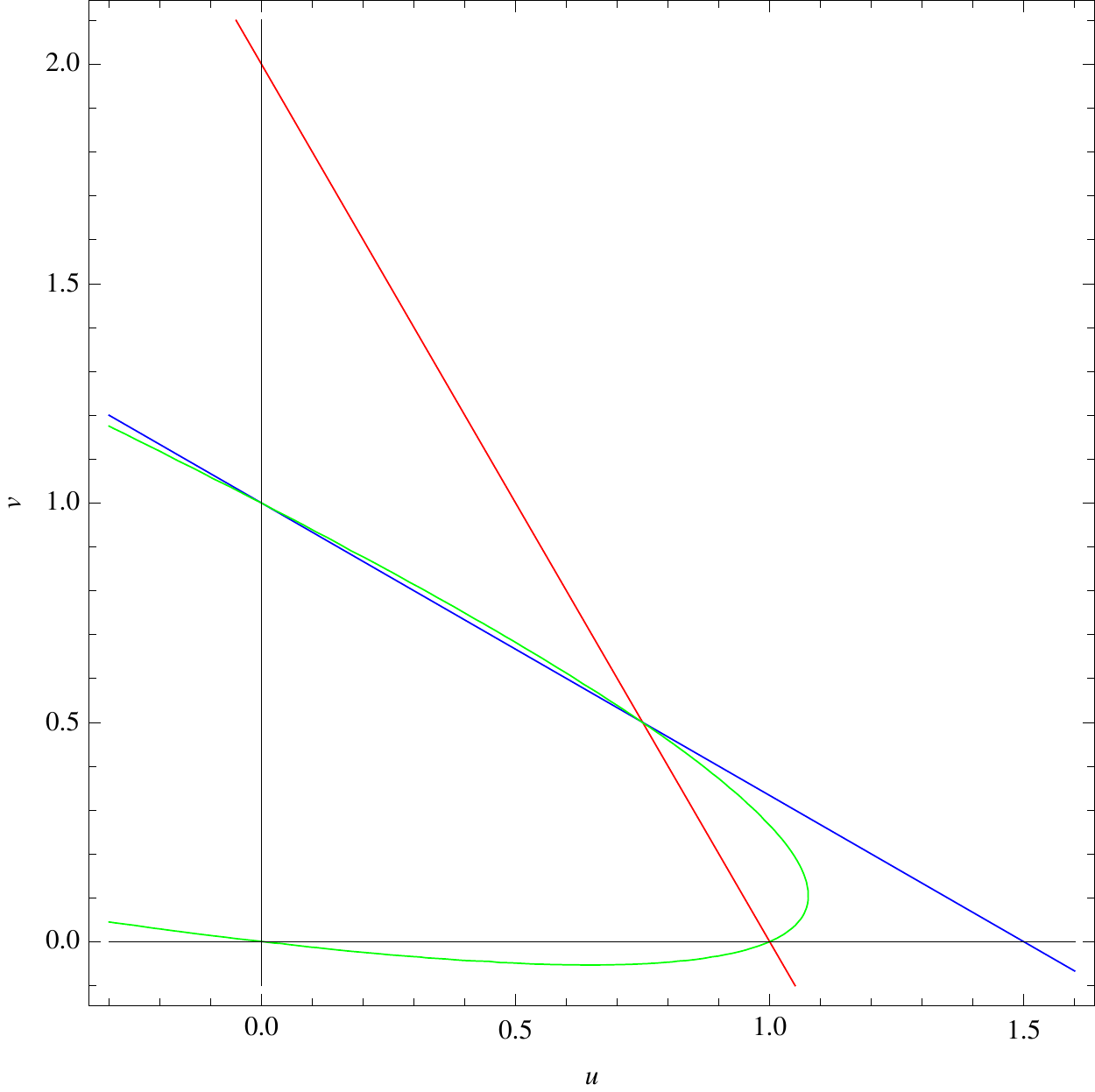}
             \label{fig: parabola_1}    } \quad \hspace{0mm}
\subfigure[]{\includegraphics[width=0.395\textwidth]{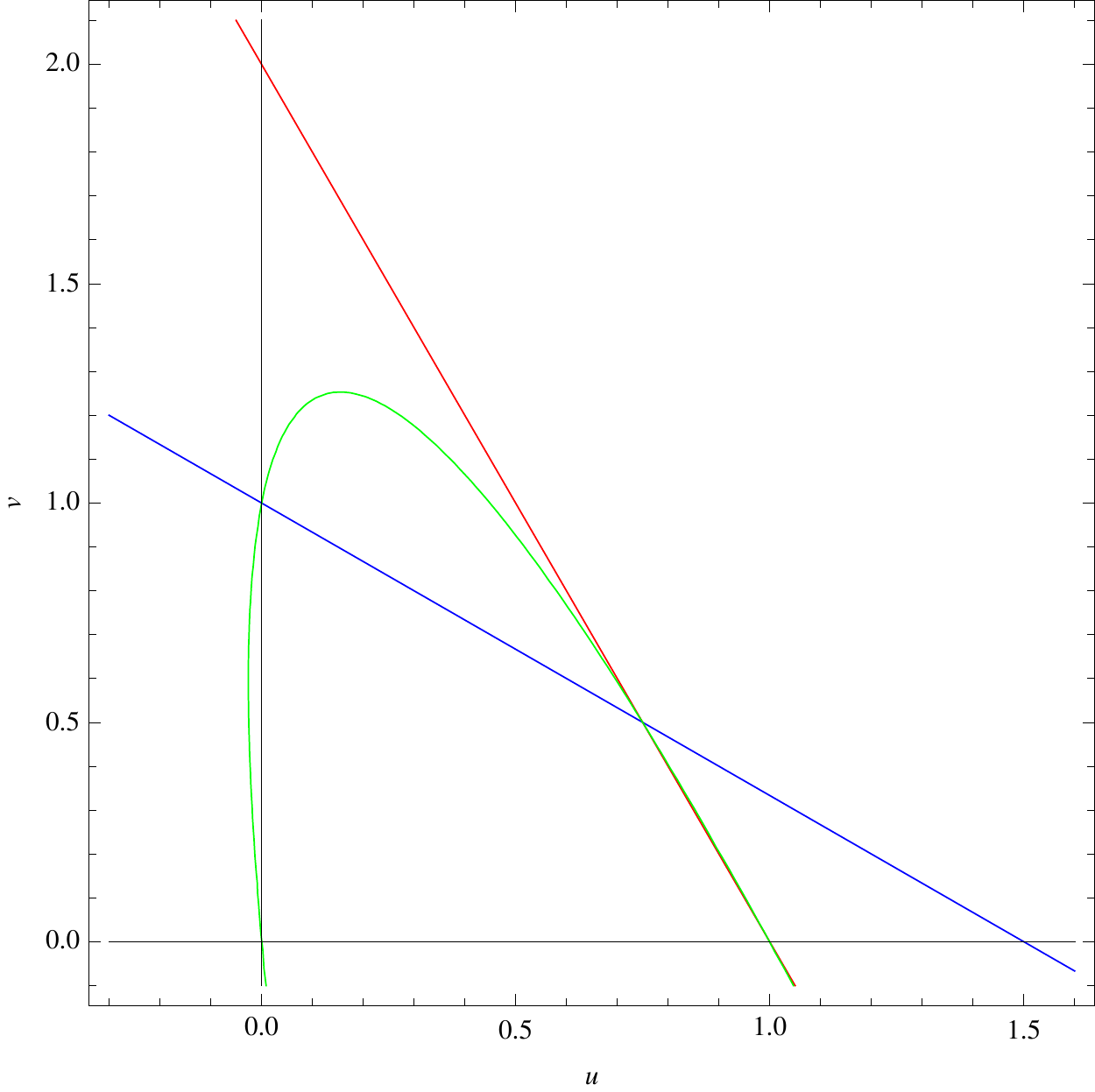}
             \label{fig: parabola_2}    } \quad \hspace{0mm}
}
\mbox{
\subfigure[]{\includegraphics[width=0.395\textwidth]{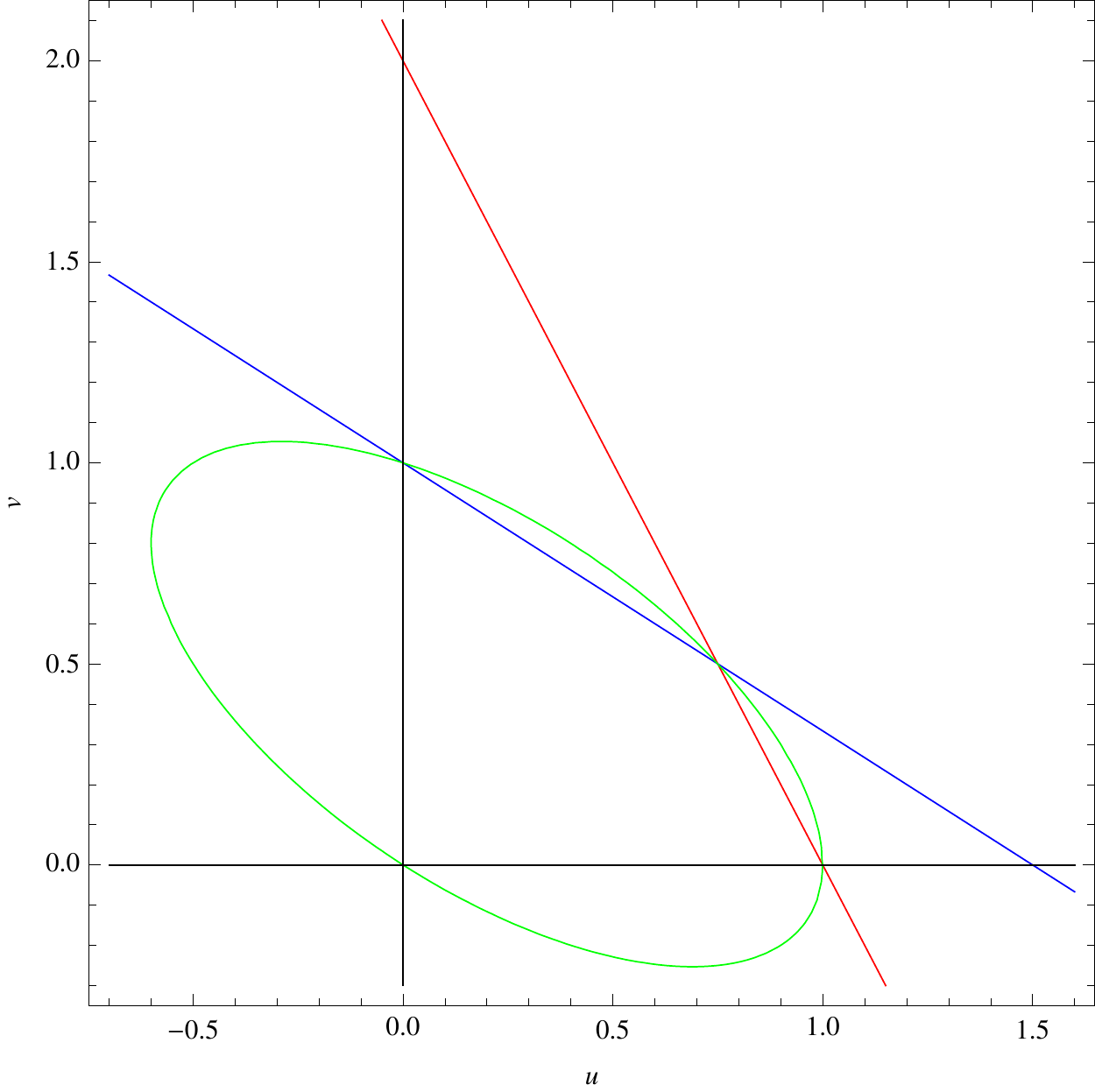}
             \label{fig: ellipse}    } \quad \hspace{0mm}
\subfigure[]{\includegraphics[width=0.395\textwidth]{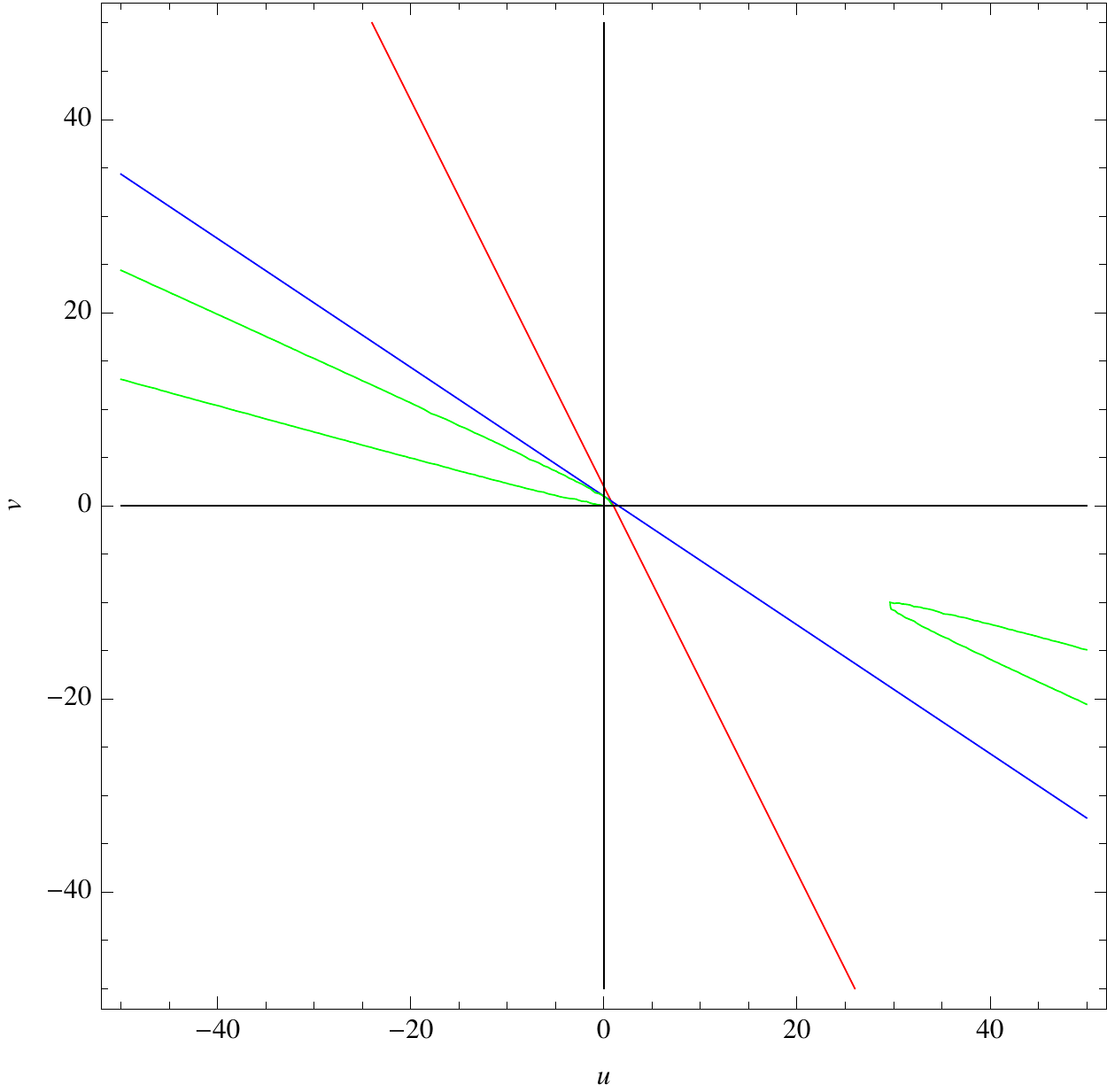}
             \label{fig: hyperbola_3_L_2}    } \quad \hspace{0mm}
}
\caption{\small Red line: $\sigma_1-c_{11}\,u-c_{12}\,v=0$; blue line: $\sigma_2-c_{21}\,u-c_{22}\,v=0$; green curve: $\alpha\,u\,(\sigma_1-c_{11}\,u-c_{12}\,v)+\beta\,v\,(\sigma_2-c_{21}\,u-c_{22}\,v)=0$. $\sigma_1=\sigma_2=c_{11}=c_{22}=1,c_{12}=\frac{1}{2},c_{21}=\frac{2}{3}$.
\subref{fig: hyperbola_3_S} $\alpha=\frac{1}{2}$, $\beta=4$ (hyperbola).
\subref{fig: hyperbola_2} $\alpha=2$, $\beta=\frac{3}{20}$ (hyperbola).
\subref{fig: parabola_1} $\alpha=2$, $\beta=\frac{15}{2}+3 \sqrt{6}\approx14.8485$ (parabola). 
\subref{fig: parabola_2} $\alpha=2$, $\beta=\frac{15}{2}-3 \sqrt{6}\approx0.1515$ (parabola).
\subref{fig: ellipse} $\alpha=2$, $\beta=3$ (ellipse).
\subref{fig: hyperbola_3_L_2} zooming out of (a).
\label{fig: 6 figures of F(u,v)=0}
}
\end{figure}

\end{proof}


\vspace{5mm}

\begin{proof}[Proof of Theorem~\ref{thm: lb<q(x)<ub}]
An easy observation shows that, it suffices to prove for any $\alpha, \beta>0$, we have
\begin{equation}
q_{\ast}
\leq \alpha\,d_1\,u(x)+\beta\,d_2\,v(x)\leq 
q^{\ast}, \quad x\in\mathbb{R},
\end{equation}
where 
\begin{equation}\label{eqn: lower bed of q}
q_{\ast}=\min\bigg[\alpha\,d_1\,\min\Big[\frac{\sigma_1}{c_{11}},\frac{\sigma_2}{c_{21}}\Big],\beta\,d_2\,\min\Big[\frac{\sigma_2}{c_{22}},\frac{\sigma_1}{c_{12}}\Big]\bigg]\min\Big[\frac{d_1}{d_2},\frac{d_2}{d_1}\Big]
\end{equation}
and
\begin{equation}\label{eqn: upper bed of q}
q^{\ast}=\max\bigg[\alpha\,d_1\,\max\Big[\frac{\sigma_1}{c_{11}},\frac{\sigma_2}{c_{21}}\Big],\beta\,d_2\,\max\Big[\frac{\sigma_2}{c_{22}},\frac{\sigma_1}{c_{12}}\Big]\bigg]\max\Big[\frac{d_1}{d_2},\frac{d_2}{d_1}\Big].
\end{equation}

First of all, we prove $(I)$ for the case of strong competition $\mathbf{[S]}$: $\frac{\sigma_1}{c_{11}}>\frac{\sigma_2}{c_{21}}$ and $\frac{\sigma_2}{c_{22}}>\frac{\sigma_1}{c_{12}}$. Clearly, \eqref{eqn: lower bed of q} in this case gives
\begin{equation}\label{eqn: lower bound of q for strong competition}
q(x)\geq \min\bigg[\alpha\,d_1\,\frac{\sigma_2}{c_{21}},\beta\,d_2\,\frac{\sigma_1}{c_{12}}\bigg]\min\Big[\frac{d_1}{d_2},\frac{d_2}{d_1}\Big], \quad x\in \mathbb{R}.
\end{equation} 
The proof of the above inequality is divided into the following four cases:
\begin{itemize}
     \item If $d_2\geq d_1$,
\begin{itemize}
\item [$(i)$] when $\beta\,\sigma_1\,c_{21}\,d_2\geq \alpha\,\sigma_2\,c_{12}\,d_1$ and $(\lambda_1,\lambda_2,\eta)=\big(\frac{\alpha\,\sigma_2\,d_1^2}{c_{21}\,d_2},\frac{\alpha\,\sigma_2\,d_1}{c_{21}},\frac{\alpha\,\sigma_2\,d_1}{c_{21}\,d_2}\big)$, $q(x)\geq \lambda_1$, $x\in \mathbb{R}$;

\item [$(ii)$] when $\beta\,\sigma_1\,c_{21}\,d_2< \alpha\,\sigma_2\,c_{12}\,d_1$ and $(\lambda_1,\lambda_2,\eta)=\big(\frac{\beta\,\sigma_1\,d_1}{c_{12}},\frac{\beta\,\sigma_1\,d_2}{c_{12}},\frac{\beta\,\sigma_1}{c_{12}}\big)$, $q(x)\geq \lambda_1$, $x\in \mathbb{R}$.
\end{itemize}
      \item If $d_2<d_1$,
\begin{itemize}
\item [$(iii)$] when $\beta\,\sigma_1\,c_{21}\,d_2\geq \alpha\,\sigma_2\,c_{12}\,d_1$ and $(\lambda_1,\lambda_2,\eta)=\big(\frac{\alpha\,\sigma_2\,d_2}{c_{21}},\frac{\alpha\,\sigma_2\,d_1}{c_{21}},\frac{\alpha\,\sigma_2}{c_{21}}\big)$, $q(x)\geq \lambda_1$, $x\in \mathbb{R}$;

\item [$(iv)$] when $\beta\,\sigma_1\,c_{21}\,d_2< \alpha\,\sigma_2\,c_{12}\,d_1$ and $(\lambda_1,\lambda_2,\eta)=\big(\frac{\beta\,\sigma_1\,d_2^2}{c_{12}\,d_1},\frac{\beta\,\sigma_1\,d_2}{c_{12}},\frac{\beta\,\sigma_1\,d_2}{c_{12}\,d_1}\big)$, $q(x)\geq \lambda_1$, $x\in \mathbb{R}$.
\end{itemize}

\end{itemize}
We first observe that the four cases can be reduced to the following two cases:
\begin{itemize}
  \item for $\beta\,\sigma_1\,c_{21}\,d_2\geq \alpha\,\sigma_2\,c_{12}\,d_1$, $q(x)\geq  \frac{\alpha\,\sigma_2\,d_1}{c_{21}}\min\big[\frac{d_1}{d_2},\frac{d_2}{d_1}\big]$, $x\in \mathbb{R}$;
  \item for $\beta\,\sigma_1\,c_{21}\,d_2< \alpha\,\sigma_2\,c_{12}\,d_1$, $q(x)\geq \frac{\beta\,\sigma_1\,d_2}{c_{12}}\min\big[\frac{d_1}{d_2},\frac{d_2}{d_1}\big]$, $x\in \mathbb{R}$.
\end{itemize}
Combining the two cases above leads to
\begin{equation}
q(x)\geq \min\bigg[\alpha\,d_1\,\frac{\sigma_2}{c_{21}},\beta\,d_2\,\frac{\sigma_1}{c_{12}}\bigg]\min\Big[\frac{d_1}{d_2},\frac{d_2}{d_1}\Big], \quad x\in \mathbb{R},
\end{equation} 
which is the desired result. The two inequalities in \eqref{eqn: L-V <0} give
\begin{equation}\label{eqn: ODE for p and q<}
q''(x)+\theta\,p'(x)+F(u(x),v(x))\leq0, \quad x\in \mathbb{R}.
\end{equation}
For $d_2>d_1$, we prove $(i)$ by contradiction. Suppose that, contrary to our claim, there exists $z\in\mathbb{R}$ such that $q(z)<\lambda_1$ and $\min_{x\in\mathbb{R}} q(x)=q(z)$. 
Since $u,v\in C^2(\mathbb{R})$, we denote respectively by $z_2$ and $z_1$ the first point intersecting the line $\alpha\,d_1\,u+\beta\,d_2\,v=\lambda_2$ in the $uv$-plane,  when the solution $(u(x),v(x))$ in the $uv$-plane travels from $z$ towards $\infty$ and $-\infty$ (as shown in Figure~3.\ref{fig: d>1_beta_a_2d>alpha_a1}). For the case where $\theta\leq0$, we integrate \eqref{eqn: ODE for p and q<} with respect to $x$ from $z_1$ to $z$ and obtain
\begin{equation}\label{eqn: integrating eqn}
q'(z)-q'(z_1)+\theta\,(p(z)-p(z_1))+\int_{z_1}^{z}F(u(x),v(x))\,dx\leq0.
\end{equation}
On the other hand we conclude: 
\begin{itemize}
  \item since $\min_{x\in\mathbb{R}} q(x)=q(z)$, 
  it is easy to see $q'(z)=\alpha\,d_1\,u'(z)+\beta\,d_2\,v'(z)=0$;
  \item $q(z_1)=\lambda_2$ follows from the fact that $z_1$ is on the line $\alpha\,d_1\,u+\beta\,d_2\,v=\lambda_2$. On the other hand, when $z_1$ moves a little towards $\infty$, $q(z_1+)$ is below the line $\alpha\,d_1\,u+\beta\,d_2\,v=\lambda_2$. This leads to $q(z_1+\delta)\leq \lambda_2$ for any $\delta>0$, and hence $q'(z_1)\leq 0$;
  \item $p(z)<\eta$ since $z$ is below the line $\alpha\,u+\beta\,v=\eta$; $p(z_1)>\eta$ since $z$ is above the line $\alpha\,u+\beta\,v=\eta$;
  \item it is readily seen that the quadratic curve $F(u,v)=0$ passes through the points $(0,0)$, $\big(\frac{\sigma_1}{c_{11}},0\big)$, $\big(0,\frac{\sigma_2}{c_{22}}\big)$, and $(u^{\ast},v^{\ast})$ in the $uv$-plane. Applying Lemma~\ref{lem: quadratic curve of F(u,v)=0}, it follows that $F(u,v)=0$ is a hyperbola and $\int_{z_1}^{z}F(u(x),v(x))\,dx>0$ since $F(u,v)<0$ as $u,v>0$ are sufficiently large.
\end{itemize}
Summarizing the above arguments, we obtain
\begin{equation}
q'(z)-q'(z_1)+\theta\,(p(z)-p(z_1))+\int_{z_1}^{z}F(u(x),v(x))\,dx>0,
\end{equation}
which contradicts \eqref{eqn: integrating eqn}. Therefore when $\theta\leq0$, $q(x)\geq \lambda_1$ for $x\in \mathbb{R}$. For the case where $\theta\geq0$, integrating \eqref{eqn: ODE for p and q<} with respect to $x$ from $z$ to $z_2$ yields
\begin{equation}\label{eqn: eqn by integrate from z to z2}
q'(z_2)-q'(z)+\theta\,(p(z_2)-p(z))+\int_{z}^{z_2}F(u(x),v(x))\,dx\leq0.
\end{equation}
In a similar manner, it can be shown that $q'(z_2)>0$, $q'(z)=0$, $p(z_2)>\eta$, $p(z)<\eta$, and $\int_{z}^{z_2}F(u(x),v(x))\,dx>0$. These together contradict \eqref{eqn: eqn by integrate from z to z2}. Consequently, $(i)$ is proved for $d_2>d_1$. 
For $d_1=d_2$, we have $q(x)=p(x)$ and \eqref{eqn: ODE for p and q<} becomes
\begin{equation}\label{eqn: ODE for p and q p=q}
p''(x)+\theta\,p'(x)+F(u(x),v(x))\leq0, \quad x\in \mathbb{R}.
\end{equation}
Moreover, when $d_1=d_2$ we take $\frac{\lambda_1}{d_1}=\frac{\lambda_2}{d_1}=\eta=\frac{\alpha\,\sigma_2}{c_{21}}$, i.e. the three lines $\alpha\,d_1\,u+\beta\,d_2\,v=\lambda_1$, $\alpha\,d_1\,u+\beta\,d_2\,v=\lambda_2$, and $\alpha\,u+\beta\,v=\eta$ coincide. Analogously to the case of $d_2>d_1$, we assume that there exists $\hat{z}\in\mathbb{R}$ such that $p(\hat{z})<\lambda_1$ and $\min_{x\in\mathbb{R}} p(x)=p(\hat{z})$. Due to $\min_{x\in\mathbb{R}} p(x)=p(\hat{z})$, we have $p'(\hat{z})=0$ and $p''(\hat{z})\geq0$. By means of Lemma~\ref{lem: quadratic curve of F(u,v)=0}, $F(u(\hat{z}),v(\hat{z}))>0$. These together give $p''(\hat{z})+\theta\,p'(\hat{z})+F(u(\hat{z}),v(\hat{z}))>0$, which contradicts \eqref{eqn: ODE for p and q p=q}. Thus, $p(x)\geq \lambda_1$ for $x\in \mathbb{R}$ when $d_1=d_2$. As a result, the proof of $(i)$ is completed.


Clearly, we see from Figures~3.\ref{fig: d>1_beta_a_2d<alpha_a1}, 3.\ref{fig: d<1_beta_a_2d>alpha_a1}, and 3.\ref{fig: d<1_beta_a_2d<alpha_a1} that the proofs of cases $(ii)$, $(iii)$, and $(iv)$ follow in a similar manner. This completes the proof of the case of strong competition $\mathbf{[S]}$: $\frac{\sigma_1}{c_{11}}>\frac{\sigma_2}{c_{21}}$ and $\frac{\sigma_2}{c_{22}}>\frac{\sigma_1}{c_{12}}$.

\begin{figure}[ht!]
\centering
\mbox{
\subfigure[]{\includegraphics[width=0.44\textwidth]{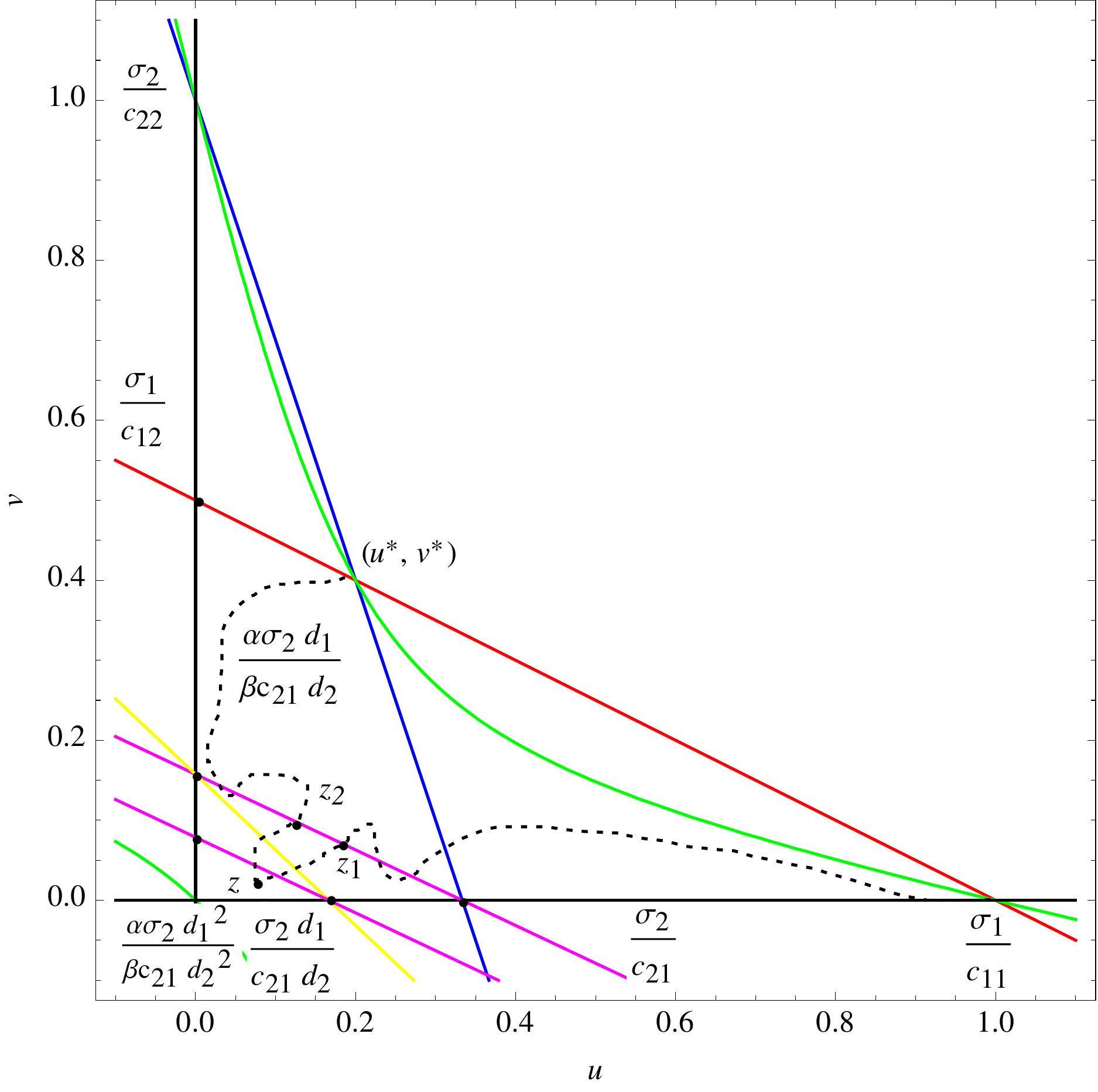}
             \label{fig: d>1_beta_a_2d>alpha_a1}    } \quad \hspace{0mm}
\subfigure[]{\includegraphics[width=0.44\textwidth]{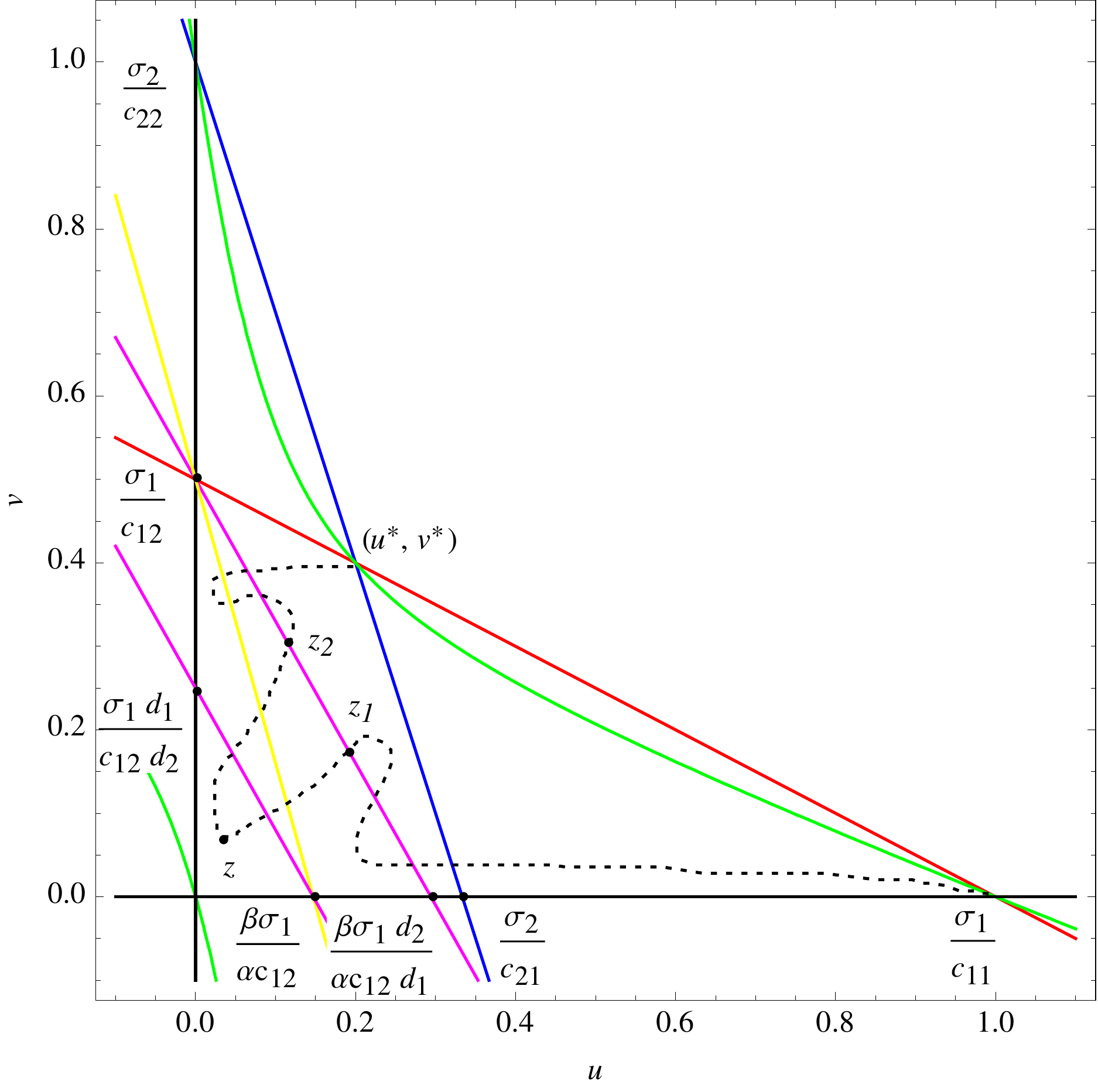}
             \label{fig: d>1_beta_a_2d<alpha_a1}    } \quad \hspace{0mm}
             }
\mbox{
\subfigure[]{\includegraphics[width=0.44\textwidth]{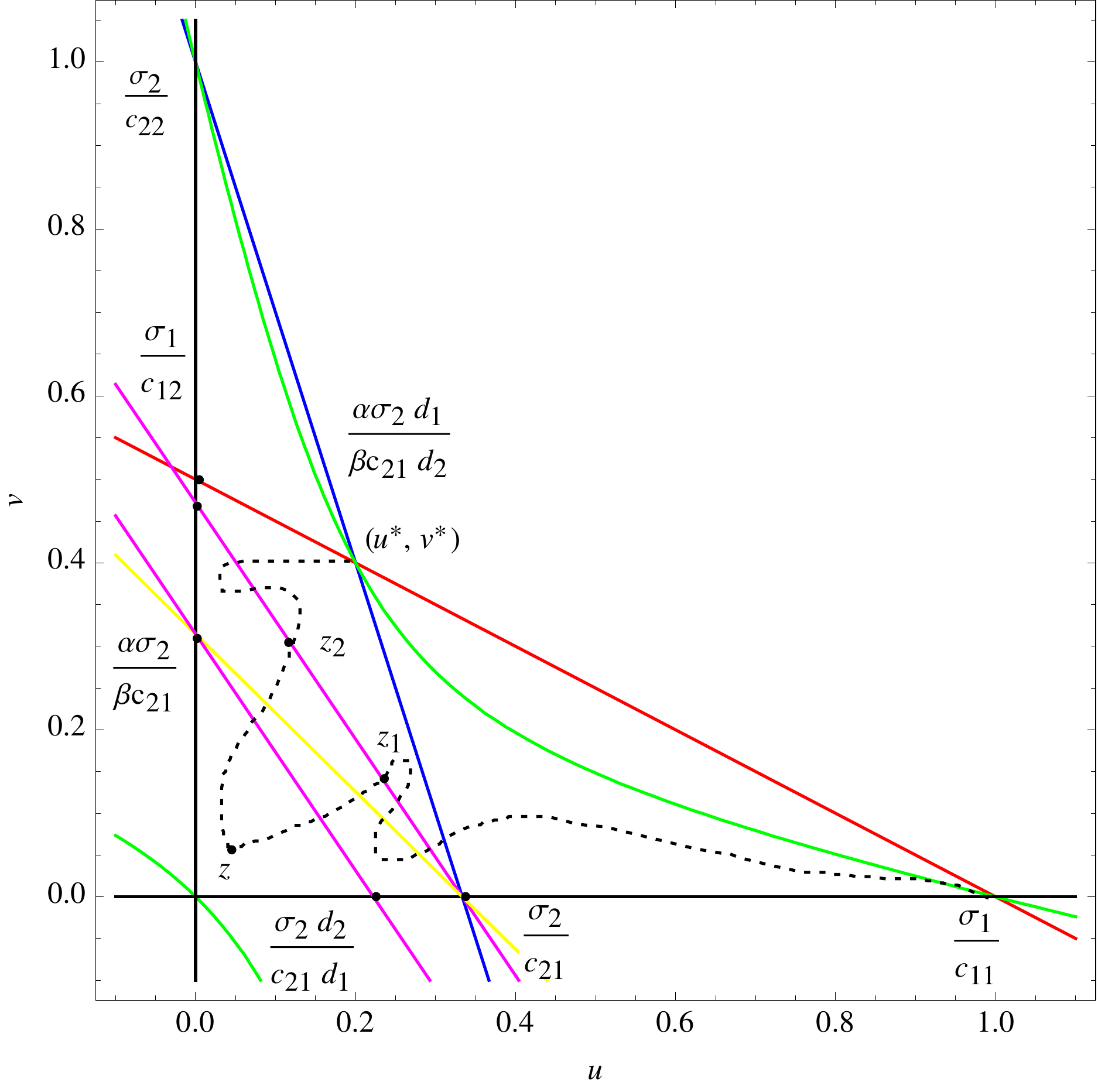}
             \label{fig: d<1_beta_a_2d>alpha_a1}    } \quad \hspace{0mm}
\subfigure[]{\includegraphics[width=0.44\textwidth]{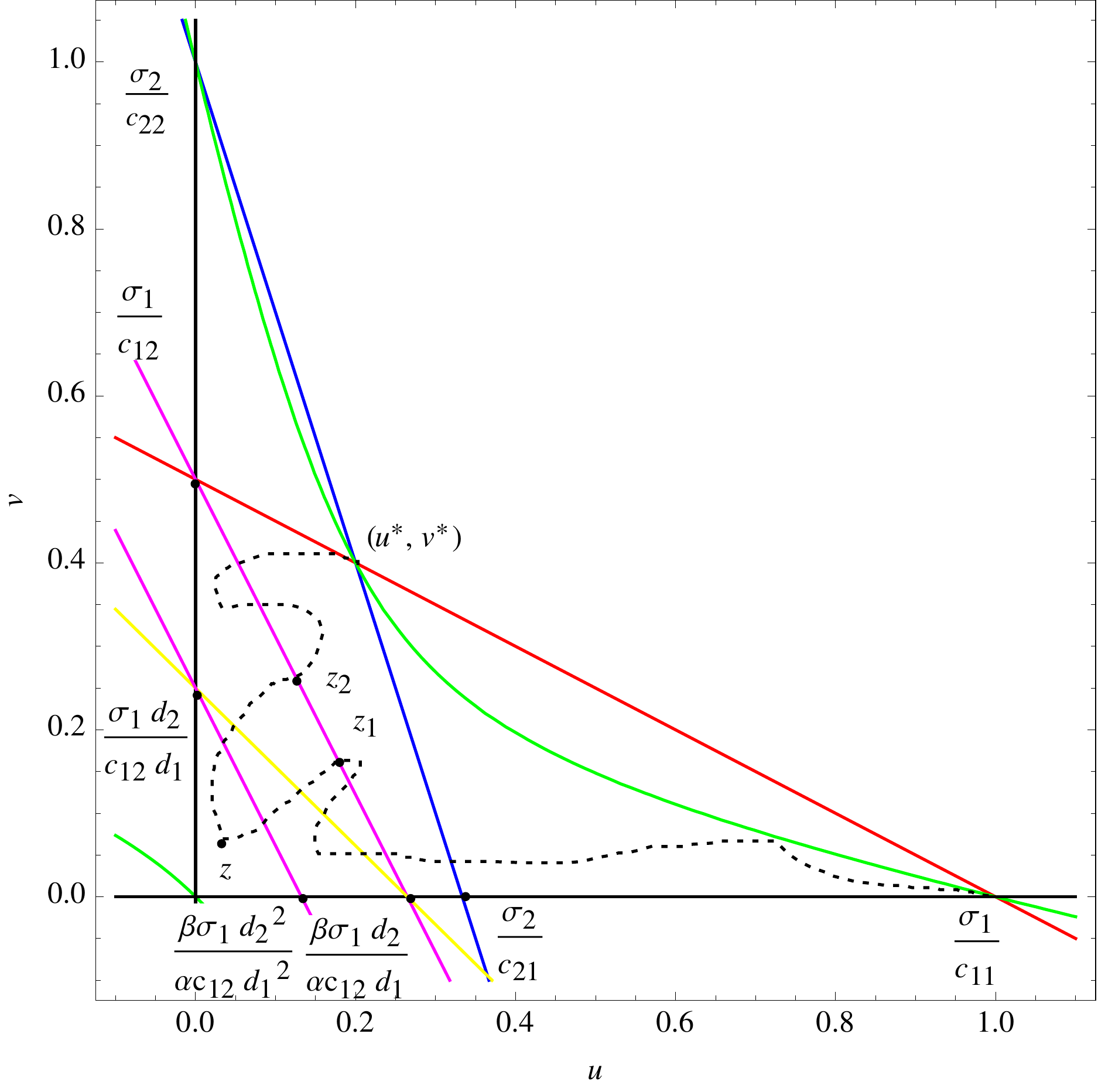}
             \label{fig: d<1_beta_a_2d<alpha_a1}    } \quad \hspace{0mm}
}
\caption{\small Red line: $\sigma_1-c_{11}\,u-c_{12}\,v=0$; blue line: $\sigma_2-c_{21}\,u-c_{22}\,v=0$; green curve: $\alpha\,u\,(\sigma_1-c_{11}\,u-c_{12}\,v)+\beta\,v\,(\sigma_2-c_{21}\,u-c_{22}\,v)=0$; magenta line (above): $\alpha\,d_1\,u+\beta\,d_2\,v=\lambda_2$; magenta line (below): $\alpha\,d_1\,u+\beta\,d_2\,v=\lambda_1$; yellow line: $\alpha\,u+\beta\,v=\eta$; dashed curve: $(u(x),v(x))$. $d_1=\sigma_1=\sigma_2=c_{11}=c_{22}=1$.
\subref{fig: d>1_beta_a_2d>alpha_a1} $c_{12}=2$, $c_{21}=3$, $\alpha=17$, $\beta=18$, and $d_2=2$ give $\lambda_1=\frac{17}{6}$, $\lambda_2=\frac{17}{3}$, and $\eta=\frac{17}{6}$.
\subref{fig: d>1_beta_a_2d<alpha_a1} $c_{12}=2$, $c_{21}=3$, $\alpha=17$, $\beta=5$, and $d_2=2$ give $\lambda_1=\frac{5}{2}$, $\lambda_2=5$, and $\eta=\frac{5}{2}$.
\subref{fig: d<1_beta_a_2d>alpha_a1} $c_{12}=2$, $c_{21}=3$, $\alpha=17$, $\beta=18$, and $d_2=\frac{2}{3}$ give $\lambda_1=\frac{34}{9}$, $\lambda_2=\frac{17}{3}$, and $\eta=\frac{17}{3}$. 
\subref{fig: d<1_beta_a_2d<alpha_a1} $c_{12}=2$, $c_{21}=3$, $\alpha=17$, $\beta=18$, and $d_2=\frac{1}{2}$ give $\lambda_1=\frac{9}{4}$, $\lambda_2=\frac{9}{2}$, and $\eta=\frac{9}{2}$.
\label{fig: 4 figures lower bound}}
\end{figure} 

To give the proof for the case of weak competition $\mathbf{[W]}$: $\frac{\sigma_1}{c_{11}}<\frac{\sigma_2}{c_{21}}$ and $\frac{\sigma_2}{c_{22}}<\frac{\sigma_1}{c_{12}}$, which leads to  
\begin{equation}\label{eqn: lower bound of q for weak competition}
q(x)\geq \min\bigg[\alpha\,d_1\,\frac{\sigma_1}{c_{11}},\beta\,d_2\,\frac{\sigma_2}{c_{22}}\bigg]\min\Big[\frac{d_1}{d_2},\frac{d_2}{d_1}\Big], \quad x\in \mathbb{R}
\end{equation} 
by \eqref{eqn: lower bed of q}, we first see from Lemma~\ref{lem: quadratic curve of F(u,v)=0} that 
$F(u,v)=0$ in this case is a hyperbola (see for example Figure~3.\ref{fig: hyperbola_3_S}, Figure~3.\ref{fig: hyperbola_2}, and Figure~3.\ref{fig: hyperbola_3_L_2}) as in the case of strong competition $\mathbf{[S]}$, a parabola (see for example Figure~3.\ref{fig: parabola_1} and Figure~3.\ref{fig: parabola_2}), or an ellipse (see for example Figure~3.\ref{fig: ellipse}) depending on the parameters in $F(u,v)=0$. However, since we are only concerned with the curve $F(u,v)=0$ in the first quadrant of the $uv$-plane, it is readily seen from Figure~\ref{fig: 6 figures of F(u,v)=0} that for each of the three generic types of quadratic curves, we can construct an \textit{N-barrier} for which the arguments used in proving the case of strong competition $\mathbf{[S]}$ remain valid. Moreover, in addition to the diffusion rates $d_1$, $d_2$ and the coefficients $\alpha$, $\beta$, the lower bound of $q(x)$ given by \eqref{eqn: lower bound of q for strong competition} under $\mathbf{[S]}$ is only involved with the minimal $u$-intercept of the lines $\sigma_1-c_{11}\,u-c_{12}\,v=0$ and $\sigma_2-c_{21}\,u-c_{22}\,v=0$, i.e. $u=\frac{\sigma_2}{c_{21}}$, and the minimal $v$-intercept of the lines $\sigma_1-c_{11}\,u-c_{12}\,v=0$ and $\sigma_2-c_{21}\,u-c_{22}\,v=0$, i.e. $v=\frac{\sigma_1}{c_{12}}$. Accordingly, for the case of weak competition $\mathbf{[W]}$, we conclude that \eqref{eqn: lower bound of q for weak competition} holds
since the minimal $u$-intercept of the lines $\sigma_1-c_{11}\,u-c_{12}\,v=0$ and $\sigma_2-c_{21}\,u-c_{22}\,v=0$ is $u=\frac{\sigma_1}{c_{11}}$, and the minimal $v$-intercept of the lines $\sigma_1-c_{11}\,u-c_{12}\,v=0$ and $\sigma_2-c_{21}\,u-c_{22}\,v=0$ is $v=\frac{\sigma_2}{c_{22}}$, respectively. For the case of strong competition $\mathbf{[S]}$, \eqref{eqn: lower bound of q for strong competition} is proved, whereas for the case of weak competition $\mathbf{[W]}$, we obtain \eqref{eqn: lower bound of q for weak competition}. Combining the two inequalities \eqref{eqn: lower bound of q for strong competition} and \eqref{eqn: lower bound of q for weak competition} yields the lower bound of $q(x)$ given by \eqref{eqn: lower bed of q}. This completes the proof of $(I)$.


As in the proof of $(I)$, there are also four cases for the proof of $(II)$ when the condition of strong competition $\mathbf{[S]}$ holds:

\begin{itemize}
     \item If $d_2\geq d_1$,
\begin{itemize}
\item [$(i)$] when $\beta\,\sigma_2\,c_{11}\,d_2\geq \alpha\,\sigma_1\,c_{22}\,d_1$ and $(\lambda_1,\lambda_2,\eta)=\big(\frac{\beta\,\sigma_2\,d_2^2}{c_{22}\,d_1},\frac{\beta\,\sigma_2\,d_2}{c_{22}},\frac{\sigma_2\,d_2}{c_{22}\,d_1}\big)$, $q(x)\leq \lambda_1$, $x\in \mathbb{R}$;

\item [$(ii)$] when $\beta\,\sigma_2\,c_{11}\,d_2< \alpha\,\sigma_1\,c_{22}\,d_1$ and $(\lambda_1,\lambda_2,\eta)=\big(\frac{\alpha\,\sigma_1\,d_2}{c_{11}},\frac{\alpha\,\sigma_1\,d_1}{c_{11}},\frac{\alpha\,\sigma_1}{c_{11}}\big)$, $q(x)\leq \lambda_1$, $x\in \mathbb{R}$.
\end{itemize}
      \item If $d_2<d_1$,
\begin{itemize}
\item [$(iii)$] when $\beta\,\sigma_2\,c_{11}\,d_2\geq \alpha\,\sigma_1\,c_{22}\,d_1$ and $(\lambda_1,\lambda_2,\eta)=\big(\frac{\beta\,\sigma_2\,d_1}{c_{22}},\frac{\beta\,\sigma_2\,d_2}{c_{22}},\frac{\beta\,\sigma_2}{c_{22}}\big)$, $q(x)\leq \lambda_1$, $x\in \mathbb{R}$;

\item [$(iv)$] when $\beta\,\sigma_2\,c_{11}\,d_2< \alpha\,\sigma_1\,c_{22}\,d_1$ and $(\lambda_1,\lambda_2,\eta)=\big(\frac{\alpha\,\sigma_1\,d_1^2}{c_{11}\,d_2},\frac{\alpha\,\sigma_1\,d_1}{c_{11}},\frac{\alpha\,\sigma_1\,d_1}{c_{11}\,d_2}\big)$, $q(x)\leq \lambda_1$, $x\in \mathbb{R}$.
\end{itemize}

\end{itemize}

Combining the four cases above, it immediately follows that 
\begin{itemize}
  \item for $\beta\,\sigma_2\,c_{11}\,d_2\geq \alpha\,\sigma_1\,c_{22}\,d_1$, $q(x)\leq  \frac{\beta\,\sigma_2\,d_2}{c_{22}}\max\big[\frac{d_1}{d_2},\frac{d_2}{d_1}\big]$ for $x\in \mathbb{R}$;
  \item for $\beta\,\sigma_2\,c_{11}\,d_2< \alpha\,\sigma_1\,c_{22}\,d_1$, $q(x)\leq  \frac{\alpha\,\sigma_1\,d_1}{c_{11}}\max\big[\frac{d_1}{d_2},\frac{d_2}{d_1}\big]$ for $x\in \mathbb{R}$,
\end{itemize}
and hence we have
\begin{equation}\label{eqn: upper bed of q under strong competition}
q(x)\leq \max\bigg[\alpha\,d_1\,\frac{\sigma_1}{c_{11}},\beta\,d_2\,\frac{\sigma_2}{c_{22}}\bigg]\max\Big[\frac{d_1}{d_2},\frac{d_2}{d_1}\Big], \quad x\in \mathbb{R}.
\end{equation} 
Indeed, \eqref{eqn: upper bed of q under strong competition} holds, as is readily seen by employing similar arguments as above together with the N-barriers constructed in Figures~3.\ref{fig: d>1_beta_d>alpha}, 3.\ref{fig: d>1_beta_d<alpha}, 3.\ref{fig: d<1_beta_d>alpha}, and 3.\ref{fig: d<1_beta_d<alpha}. On the other hand, under the condition of weak competition $\mathbf{[W]}$, \eqref{eqn: upper bed of q} leads to
\begin{equation}\label{eqn: upper bed of q under weak competition}
q(x)\leq \max\bigg[\alpha\,d_1\,\frac{\sigma_2}{c_{21}},\beta\,d_2\,\frac{\sigma_1}{c_{12}}\bigg]\max\Big[\frac{d_1}{d_2},\frac{d_2}{d_1}\Big], \quad x\in \mathbb{R},
\end{equation}
which can be shown as in the proof of \eqref{eqn: lower bound of q for weak competition} by interchanging the roles of $\frac{\sigma_1}{c_{11}}$ ($\frac{\sigma_2}{c_{22}}$, respectively) and $\frac{\sigma_2}{c_{21}}$ ($\frac{\sigma_1}{c_{12}}$, respectively) in \eqref{eqn: upper bed of q under strong competition}. Therefore $(II)$ of Theorem~\ref{thm: lb<q(x)<ub} follows from \eqref{eqn: upper bed of q under strong competition} and \eqref{eqn: upper bed of q under weak competition}. The proof of Theorem~\ref{thm: lb<q(x)<ub} is completed.





\begin{figure}[ht!]
\centering
\mbox{
\subfigure[]{\includegraphics[width=0.44\textwidth]{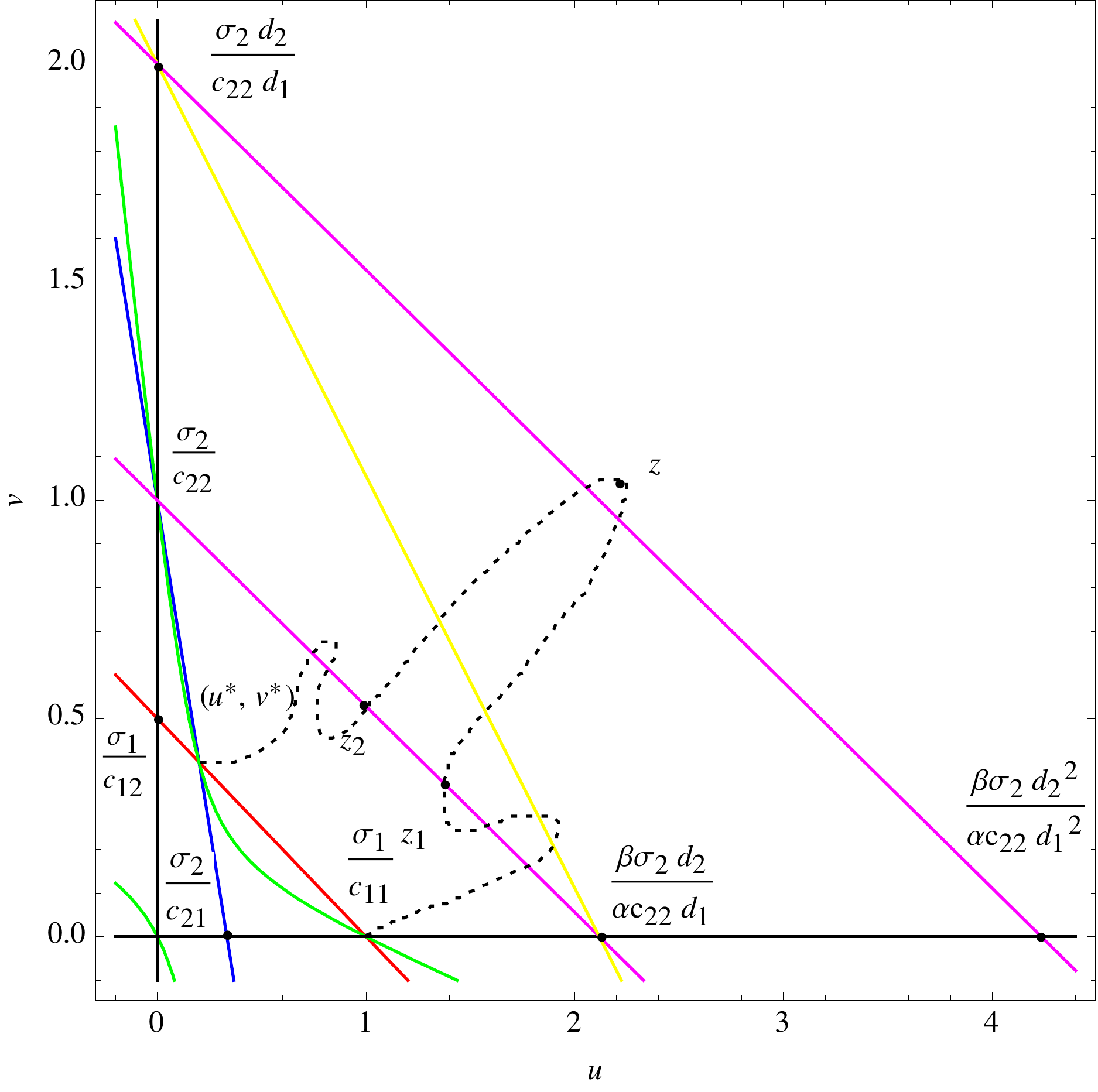}
             \label{fig: d>1_beta_d>alpha}    } \quad \hspace{0mm}
\subfigure[]{\includegraphics[width=0.44\textwidth]{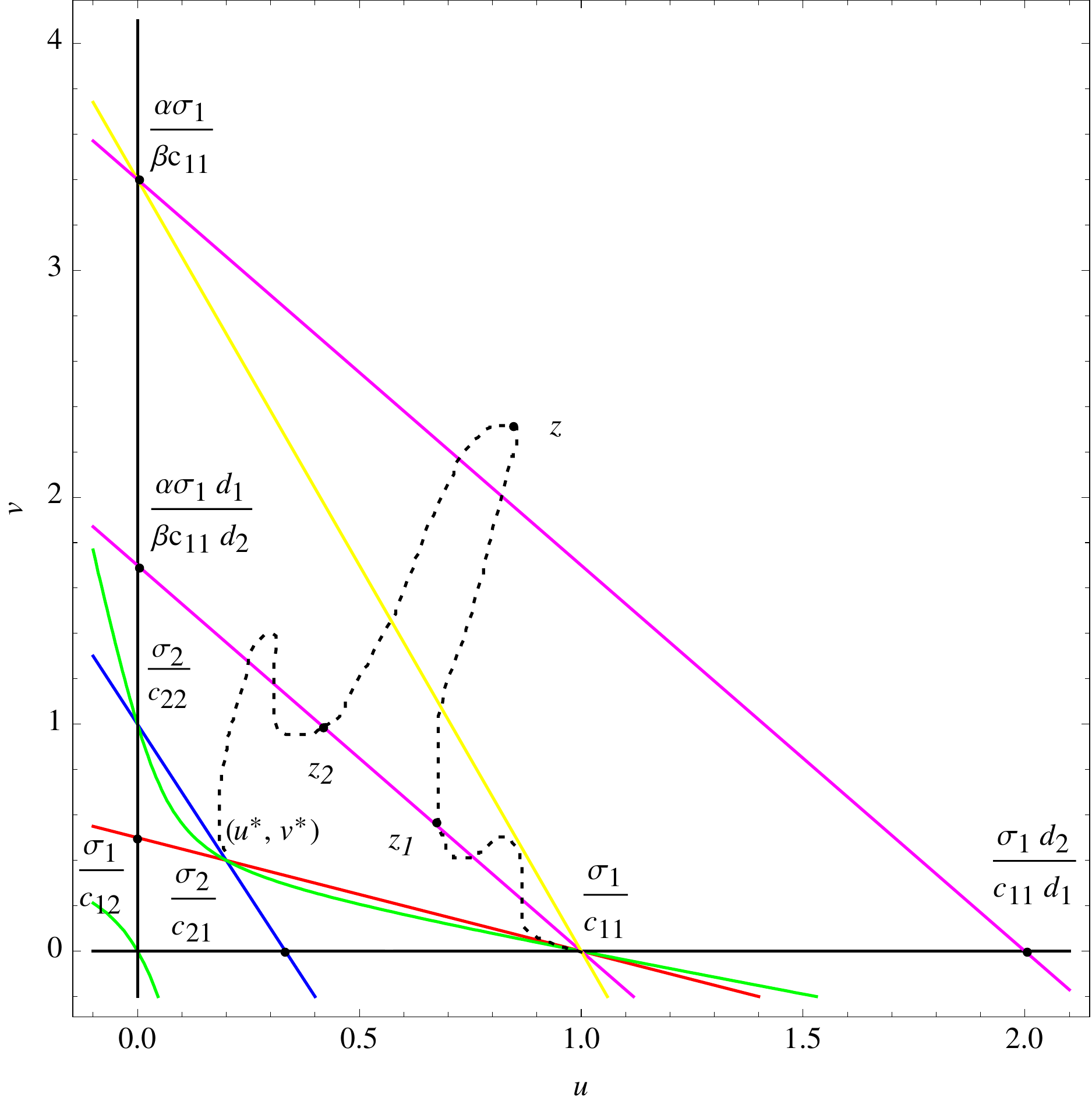}
             \label{fig: d>1_beta_d<alpha}    } \quad \hspace{0mm}
             }
\mbox{
\subfigure[]{\includegraphics[width=0.44\textwidth]{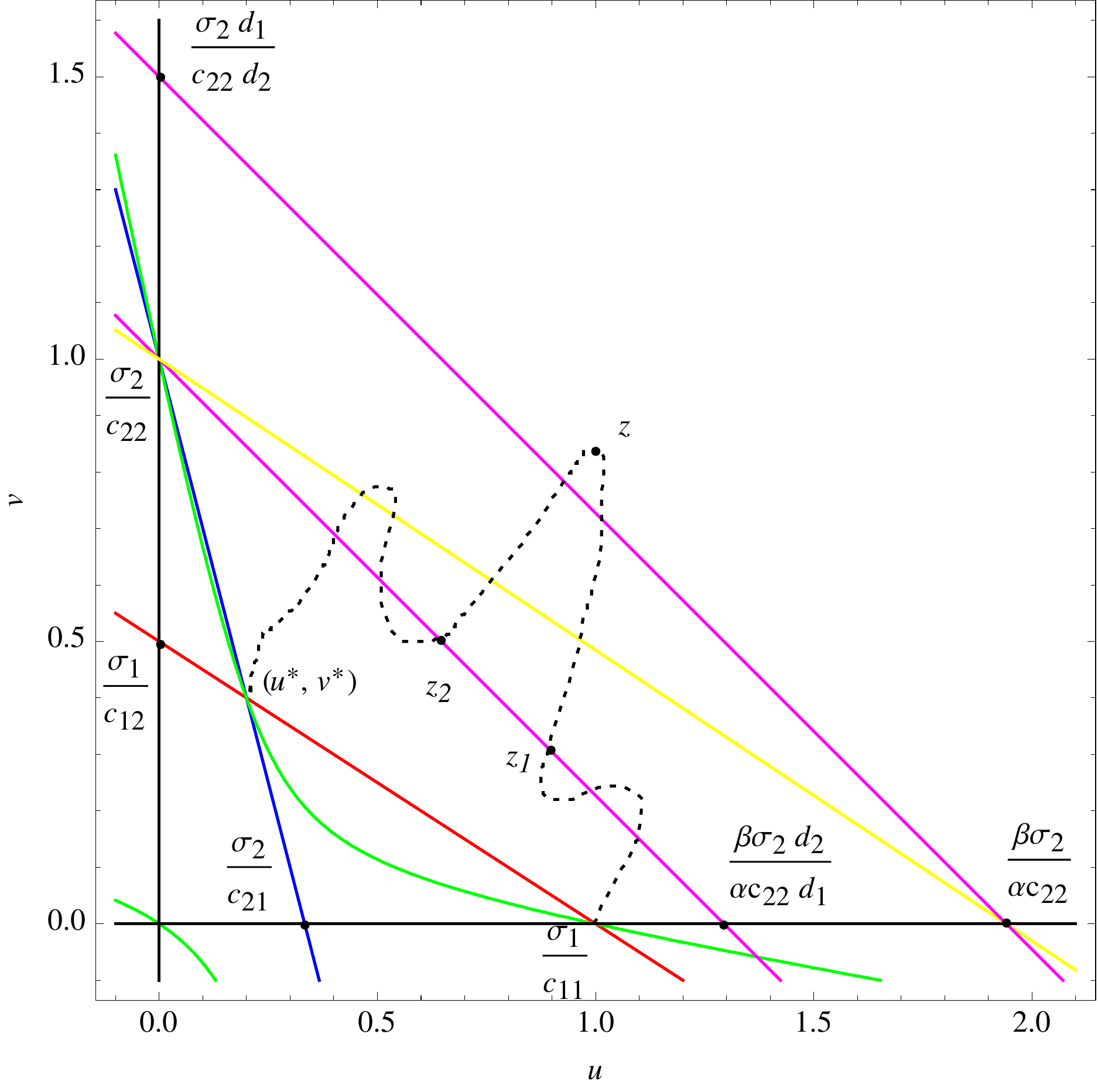}
             \label{fig: d<1_beta_d>alpha}    } \quad \hspace{0mm}
\subfigure[]{\includegraphics[width=0.44\textwidth]{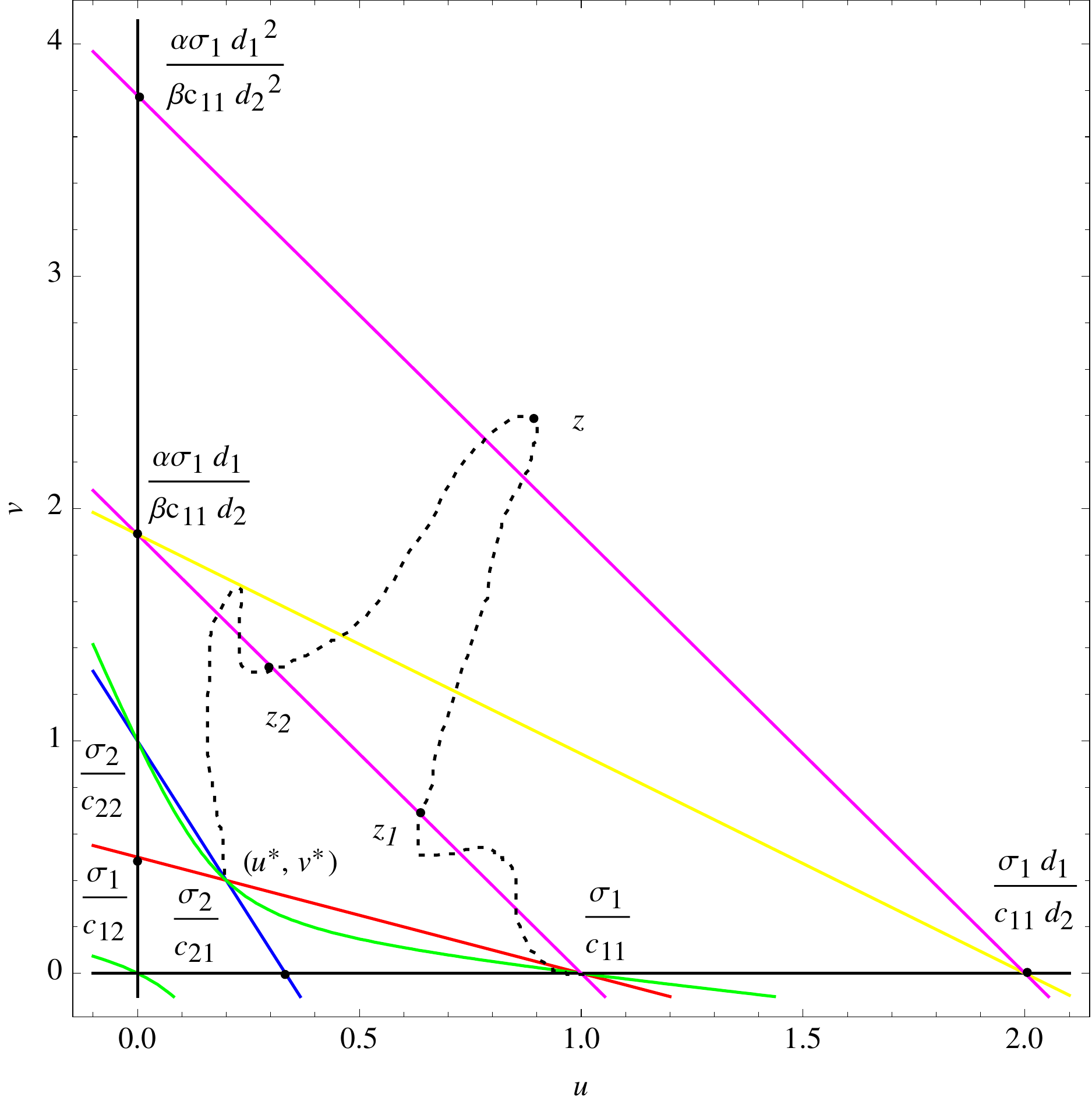}
             \label{fig: d<1_beta_d<alpha}    } \quad \hspace{0mm}
}
\caption{\small Red line: $\sigma_1-c_{11}\,u-c_{12}\,v=0$; blue line: $\sigma_2-c_{21}\,u-c_{22}\,v=0$; green curve: $\alpha\,u\,(\sigma_1-c_{11}\,u-c_{12}\,v)+\beta\,v\,(\sigma_2-c_{21}\,u-c_{22}\,v)=0$; magenta line (below): $\alpha\,d_1\,u+\beta\,d_2\,v=\lambda_2$; magenta line (above): $\alpha\,d_1\,u+\beta\,d_2\,v=\lambda_1$; yellow line: $\alpha\,u+\beta\,v=\eta$; dashed curve: $(u(x),v(x))$. $d_1=\sigma_1=\sigma_2=c_{11}=c_{22}=1$.
\subref{fig: d>1_beta_d>alpha} $c_{12}=2$, $c_{21}=3$, $\alpha=17$, $\beta=18$, and $d_2=2$ give $\lambda_1=72$, $\lambda_2=36$, and $\eta=36$.
\subref{fig: d>1_beta_d<alpha} $c_{12}=2$, $c_{21}=3$, $\alpha=17$, $\beta=5$, and $d_2=2$ give $\lambda_1=34$, $\lambda_{2}=17$, and $\eta=17$.
\subref{fig: d<1_beta_d>alpha} $c_{12}=2$, $c_{21}=3$, $\alpha=17$, $\beta=33$, and $d_2=\frac{2}{3}$ give $\lambda_1=33$, $\lambda_2=22$, and $\eta=33$. 
\subref{fig: d<1_beta_d<alpha} $c_{12}=2$, $c_{21}=3$, $\alpha=17$, $\beta=18$, and $d_2=\frac{1}{2}$ give $\lambda_1=34$, $\lambda_2=17$, and $\eta=34$.
\label{fig: 4 figures upper bound}}
\end{figure}

\end{proof}

\vspace{5mm}
\setcounter{equation}{0}
\setcounter{figure}{0}
\section{New exact $(1,0,0)$-$(u^{\ast},v^{\ast},0)$ waves}\label{sec: exact (1,0,0)-(u,v,0) waves}
\vspace{5mm}




In this section, we always assume $\sigma_1=c_{11}$, unless otherwise stated. Looking for traveling wave solutions $(u(x),v(x),w(x))$ with the profiles of $u(x)$ being decreasing in $x$, $v(x)$ being increasing in $x$, and $w(x)$ being a pulse (i.e., $w(\pm\infty)=0$ and $w(x)>0$ for $x\in \mathbb{R}$) of \eqref{eqn: L-V systems of three species (TWS)} leads to the following ans\"{a}tz (\cite{Chen&Hung15Nonexistence,CHMU-semi,CHMU,hung2012JJIAM}) for solving \eqref{eqn: L-V systems of three species (TWS)}
\begin{equation}\label{eqn: ansatz for solns}
\begin{cases}
\vspace{3mm}u(x) = \frac{\displaystyle1}{\displaystyle 2}(u^{\ast}+1)+\frac{\displaystyle1}{\displaystyle2}(u^{\ast}-1)\,T(x), \\
\vspace{3mm}v(x) = k_1\,(1+ T(x) )^2, \\
w(x) = k_2\,(1- T^2(x)),
\end{cases}
\end{equation}
where $T(x)=\tanh(x)$, $k_1=\frac{\displaystyle v^{\ast}}{\displaystyle 4}$ and $k_2$ is a positive constant to be determined. It is readily verified that the ans\"{a}tz \eqref{eqn: ansatz for solns} satisfies the boundary conditions at $x=\pm\infty$ in \eqref{eqn: L-V systems of three species (TWS)}. Since $u(x)$, $v(x)$ and $w(x)$ in \eqref{eqn: ansatz for solns} are expressed in terms of polynomials in $\tanh (x)$ and $\frac{d}{dx} \tanh (x)=1-\tanh^2 (x)$, inserting \eqref{eqn: ansatz for solns} into the three equations in \eqref{eqn: L-V systems of three species (TWS)} gives 
\begin{subequations}\label{eqn: alge eqns 10}
\begin{equation}
d_1\,u_{xx}+\theta \,u_x+u\,(\sigma_1-c_{11}\,u-c_{12}\,v-c_{13}\,w)
=\Big[\zeta_{10} + \zeta_{11} \,T(x)+\zeta_{12}  \,{T^2(x)}+ \zeta_{13}  \,{T^3(x)}  \,\Big],
\end{equation}
\begin{equation}
d_2\,v_{xx}+\theta \,v_x+v\,(\sigma_2-c_{21}\,u-c_{22}\,v-c_{23}\,w)
=v\,\Big[\zeta_{20} +\zeta_{21} \,T(x)+\zeta_{22} \,{T^2(x)}  \,\Big],
\end{equation}
\begin{equation}
d_3\,w_{xx}+\theta \,w_x+w\,(\sigma_3-c_{31}\,u-c_{32}\,v-c_{33}\,w)
=w\,\Big[\zeta_{30} +\zeta_{31} \,T(x)+ \zeta_{32} \,{T^2(x)}  \,\Big].
\end{equation}
\end{subequations}
Equating the coefficients of powers of $T(x)$ to zero yields a system of ten algebraic equations:
\begin{equation}\label{eqn: alge eqns 9}
\begin{cases}
\vspace{2mm}
\zeta_{1i}=0, \; i=0,1,2,3,\\
\vspace{2mm}
\zeta_{2i}=0,  \; i=0,1,2,\\
\zeta_{3i}=0, \; i=0,1,2.\\
\end{cases}
\end{equation}
It turns out that \eqref{eqn: alge eqns 10} can be solved to give 

\begin{subequations}\label{eqn: restrictions of parameters of exact soln}
\begin{equation}
c_{11}=\sigma_1,
\end{equation}
\begin{equation}
c_{12}=\frac{d_1 \sigma_1}{k_1 \left(2 d_1+\theta\right)},
\end{equation}
\begin{equation}
c_{13}= \frac{-2 d_1 \theta+d_1 \sigma _1-4 d_1^2}{k_2\left(2 d_1+\theta \right)},
\end{equation}
\begin{equation}
c_{21}=16 d_2+4 \theta
   +\sigma _2,
\end{equation}
\begin{equation}
c_{22}= \frac{2 d_1 \theta -4
   d_2 \theta +d_1 \sigma _2+8
   d_1 d_2-\theta ^2}{k_1
   \left(2 d_1+\theta \right)},
\end{equation}
\begin{equation}
c_{23}=\frac{2 d_1 \theta -10
   d_2 \theta +d_1 \sigma _2-4
   d_1 d_2-\theta ^2}{k_2
   \left(2 d_1+\theta \right)},
\end{equation}
\begin{equation}
c_{31}=4 d_3+2 \theta
   +\sigma _3,
\end{equation}
\begin{equation}
c_{32}=\frac{d_1 \sigma _3+4
   d_1 d_3-\theta ^2}{k_1
   \left(2 d_1+\theta \right)},
\end{equation}
\begin{equation}
c_{33}=\frac{-6 d_3 \theta
   +d_1 \sigma_3-8 d_1
   d_3-\theta ^2}{k_2 \left(2
   d_1+\theta \right)}.
\end{equation}
\end{subequations}
The result obtained is summarized in the following theorem.
\vspace{5mm}

\begin{thm}[\textbf{Exact $(1,0,0)$-$(u^{\ast},v^{\ast},0)$ waves}]\label{thm: exact (1,0,0)-(u,v,0) wave}
The boundary value problem \eqref{eqn: L-V systems of three species (TWS)} with $\sigma_1=c_{11}$ admits an exact solution of the form \eqref{eqn: ansatz for solns} provided that \eqref{eqn: restrictions of parameters of exact soln} holds.
\end{thm}
\vspace{5mm}
Theorem~\ref{thm: exact (1,0,0)-(u,v,0) wave} asserts that under certain conditions imposed on the parameters, i.e. under \eqref{eqn: restrictions of parameters of exact soln}, we can find exact $(1,0,0)$-$(u^{\ast},v^{\ast},0)$ waves of \eqref{eqn: L-V systems of three species (TWS)} and the exact waves are polynomials in $\tanh (x)$. To illustrate Theorem~\ref{thm: exact (1,0,0)-(u,v,0) wave}, let us choose $k_1=k_2=d_1=d_2=d_3=1$, $\theta=3$, and $\sigma _1=\sigma_2=\sigma_3=41$ in \eqref{eqn: restrictions of parameters of exact soln}. This gives $c_{11}=41,c_{12}=\frac{41}{5},c_{13}=\frac{31}{5},c_{21}=69,c_{22}=\frac{34}{5},c_{23}=\frac{4}{5},c_{31}=51,c_{32}=\frac{36}{5},c_{33}=\frac{6}{5}$ and $(u^{\ast},v^{\ast})=(\frac{1}{5},4)$. The resulting exact $(1,0,0)$-$(u^{\ast},v^{\ast},0)$ wave is given by
\begin{equation}
\begin{cases}
\vspace{3mm}u(x) = \frac{\displaystyle3}{\displaystyle 5}-\frac{\displaystyle2}{\displaystyle 5}\,\tanh(x), \\
\vspace{3mm}v(x) = (1+ \tanh(x) )^2, \\
w(x) = 1- \tanh^2(x).
\end{cases}
\end{equation}
The profiles of $u(x)$, $v(x)$, and $w(x)$ are shown in Figure~\ref{eqn: figure of uvw}. 
\begin{figure}[ht!]
\center \scalebox{0.8}{\includegraphics{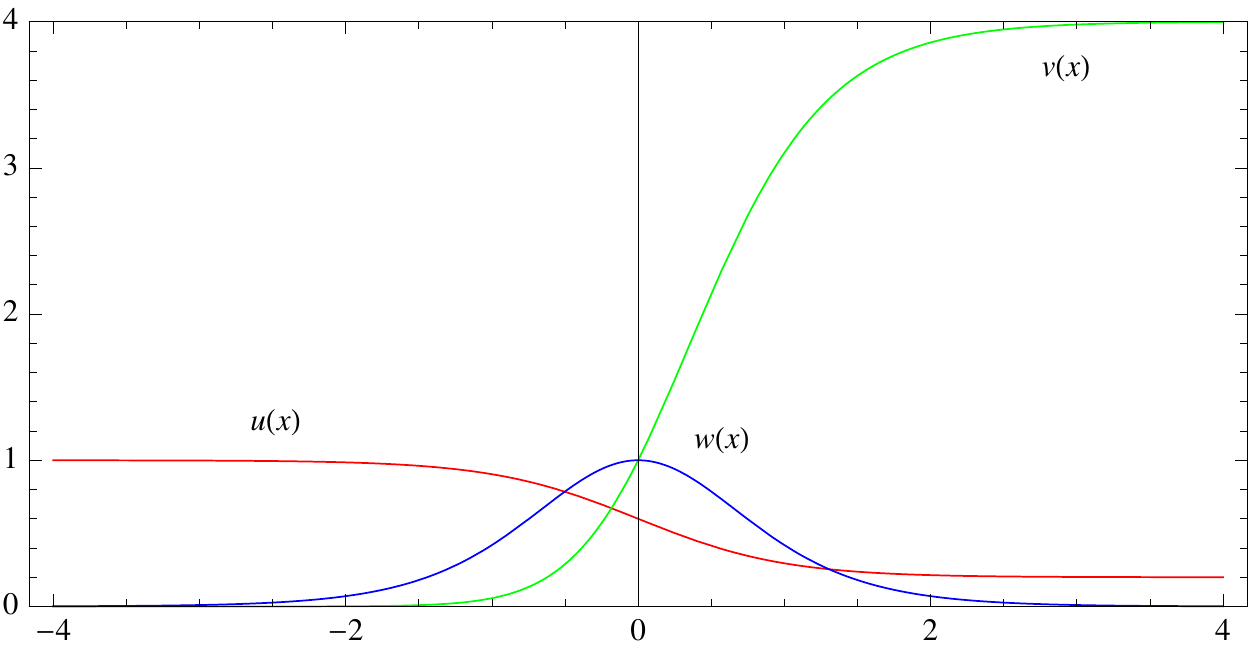}}
\caption{Profiles of the solution $(u(x),v(x),w(x))$.}
\label{eqn: figure of uvw}
\end{figure}
We conclude this section with the remark that the ans\"{a}tz \eqref{eqn: ansatz for solns} for solutions of \eqref{eqn: L-V systems of three species (TWS)} is inspired by the one proposed in \cite{hung2012JJIAM}, where the ans\"{a}tz \eqref{eqn: ansatz for solns} for solutions is
\begin{equation}
\begin{cases}
\vspace{3mm}u(x) = \frac{\displaystyle1}{\displaystyle 2}(u^{\ast}+1)+\frac{\displaystyle1}{\displaystyle2}(u^{\ast}-1)\,T(x), \\
v(x) = k_1\,(1+ T(x) )^2, \\
\end{cases}
\end{equation}
when $w$ is absent in \eqref{eqn: L-V systems of three species (TWS)}.


\vspace{5mm}
\setcounter{equation}{0}
\setcounter{figure}{0}
\section{Applications of the N-barrier maximum principle
}\label{sec: Applications of N-barrier maxi principle }
\vspace{5mm}

\subsection{Application to the existence of three species traveling waves: proof of Theorem~\ref{thm: Existence 3 species}
}\label{subsec: Application to the existence}
\vspace{5mm}


To prove Theorem~\ref{thm: Existence 3 species}, we first observe that the third equation in \eqref{eqn: L-V systems of three species (TWS) c13=c23=0} can be regarded as a non-autonomous Fisher equation when $u=u(x)$ and $v=v(x)$ are given, say $u=\tilde{u}(x)$ and $v=\tilde{v}(x)$. Then the non-autonomous Fisher equation (\cite{Berestycki09Species-shifting-climate}, \cite{Sanchez00Nonautonomous-Fisher}, \cite{Volpert99Non-autonomous-KPP}) is
\begin{equation}\label{eqn: eqn of Fisher type}
d_3\,w_{xx}+\theta \,w_x+w\,(\sigma_3-c_{31}\,\tilde{u}-c_{32}\,\tilde{v}-c_{33}\,w)=0, \quad x\in \mathbb{R}.
\end{equation}
In order to find a solution of \eqref{eqn: eqn of Fisher type}, an approach based on the supersolution-subsolution method is employed. To this end, we introduce supersolutions and subsolutions. $\bar{w}=\bar{w}(x)$ is said to be a \textit{supersolution} of \eqref{eqn: eqn of Fisher type} if it satisfies the differential inequality
\begin{equation}\label{ineq: supsolution of eqn of Fisher type}
d_3\,\bar{w}_{xx}+\theta \,\bar{w}_z+\bar{w}\,(\sigma_3-c_{31}\,\tilde{u}-c_{32}\,\tilde{v}-c_{33}\,\bar{w})\leq 0, \quad x\in \mathbb{R}.
\end{equation}
Similarly, a \textit{subsolution} $\underaccent{\bar}{w}=\underaccent{\bar}{w}(x)$ is defined by reversing the inequality in \eqref{ineq: supsolution of eqn of Fisher type}. The following lemma is helpful in constructing non-trivial solutions of \eqref{eqn: eqn of Fisher type}.
\vspace{5mm}

\begin{lem}[\cite{Berestycki09Species-shifting-climate}]\label{lem: lemma from Berestycki et al}
Suppose that $w(x)$ is a bounded solution of 
\begin{equation}
d_3\,w_{xx}+\theta \,w_x+w\,\varphi(w,x)=0, \quad x\in \mathbb{R},
\end{equation}
where $\varphi(0,x)\rightarrow -{\varphi_0}^{\pm}$ as $x\rightarrow\pm\infty$ for some constants ${\varphi_0}^{+},{\varphi_0}^{-}>0$. Then $w(x)\rightarrow0$ as $x\rightarrow\pm\infty$.
\end{lem}

\vspace{5mm}

To construct a pair of subsolution and supersolution of \eqref{eqn: eqn of Fisher type}, we employ the \textit{tanh method}.

\vspace{5mm}

\begin{proof}[Proof of Theorem~\ref{thm: Existence 3 species}]
Due to $\mathbf{[H4]}$, we clearly have $\underaccent\bar{w}(x)\leq \bar{w}(x)$ for $x\in\mathbb{R}$. First of all, we show that $\underaccent\bar{w}(x)$ and $\bar{w}(x)$ are a subsolution and a supersolution of \eqref{eqn: eqn of Fisher type} respectively. Indeed, a straightforward computation gives us
\begin{align}\label{eqn: subsoln plug into gen. Fisher= 2nd. poly. of tanhz}
&d_3\,\underaccent\bar{w}_{xx}+\theta \,\underaccent\bar{w}_z+\underaccent\bar{w}\,(\,\sigma_3-c_{31}\,\tilde{u}-c_{32}\,\tilde{v}-c_{33}\,\underaccent\bar{w})\notag\\[3mm]
&=\underaccent\bar{w}\,
\big[
-c_{33}\, \underaccent\bar{K}-2\,d_3-\bar{q}+\sigma _3
-2\,\theta\,\tanh(x)+(c_{33}\,\underaccent\bar{K}+6\,d_3) \,\tanh^2(x)  
\big]\geq 0,   \notag\\[3mm]
&d_3\,\bar{w}_{xx}+\theta \,\bar{w}_z+\bar{w}\,(\,\sigma_3-c_{31}\,\tilde{u}-c_{32}\,\tilde{v}-c_{33}\,\bar{w})\notag\\[3mm]
&=\bar{w}\,
\big(
-c_{33}\, \bar{K}-\underaccent\bar{q}+\sigma _3 
\big)\leq 0
\end{align}
under hypotheses $\mathbf{[H2]}$ and $\mathbf{[H3]}$. By means of Theorem~\ref{thm: lb<q(x)<ub}, we have used in the last two inequalities an estimate of $c_{31}\,\tilde{u}(x)+c_{32}\,\tilde{v}(x)$, i.e.
\begin{equation}
\underaccent\bar{q}\leq c_{31}\,\tilde{u}(x)+c_{32}\,\tilde{v}(x)\leq \bar{q},\quad x\in\mathbb{R}.
\end{equation}
The existence of a solution $w=w(x)$ for \eqref{eqn: eqn of Fisher type} lying between the subsolution $\underaccent\bar{w}(x)$ and the supersolution $\bar{w}(x)$ constructed above follows from Theorem 2.8 in \cite{KochMedina95longtimebehv}. In view of $\mathbf{[H1]}$, we finally employ Lemma~\ref{lem: lemma from Berestycki et al} to conclude that the solution $w(x)$ of \eqref{eqn: eqn of Fisher type} has the asymptotic behavior $\lim_{|x|\rightarrow \infty} w(x)=0$. This completes the proof.
\end{proof}

\subsection{Application to the nonexistence of three species traveling waves: proof of Theorem~\ref{thm: Nonexistence 3 species} }
\label{subsec: nonexistence}
\vspace{5mm}

\begin{proof}[Proof of Theorem~\ref{thm: Nonexistence 3 species}]
We prove Theorem~\ref{thm: Nonexistence 3 species} by contradiction.
Suppose that, to the contrary, there exist $u(x),v(x),w(x)>0$, $x\in\mathbb{R}$ satisfying \eqref{eqn: L-V systems of three species (TWS)}.  Since $w(x)>0$ for $x\in\mathbb{R}$ and $w(\pm\infty)=0$, there exists $x_0\in\mathbb{R}$ such that $\max_{x\in\mathbb{R}} w(x)=w(x_0)>0$, $w''(x_0)<0$, and $w'(x_0)=0$. Due to $d_3\,w_{xx}+\theta \,w_x+w(\sigma_3-c_{31}\,u-c_{32}\,v-c_{33}\,w)=0$, we obtain
\begin{equation}\label{eqn: w(???) ???>0}
\sigma_3-c_{31}\,u(x_0)-c_{32}\,v(x_0)-c_{33}\,w(x_0)>0,
\end{equation}
and hence
\begin{equation}
w(x)\leq w(x_0)< \frac{1}{c_{33}}\big(\sigma_3-c_{31}\,u(x_0)-c_{32}\,v(x_0)\big)<\frac{\sigma_3}{c_{33}},\;x\in\mathbb{R}.
\end{equation}
As a result, we have
\begin{equation}\label{eqn: nonexistence diff ineq <0}
\begin{cases}
\vspace{3mm}
d_1\,u_{xx}+\theta \,u_x+u(\sigma_1-c_{13}\,\sigma_3\,c_{33}^{-1}-c_{11}\,u-c_{12}\,v)\leq0, \quad x\in\mathbb{R}, \\
d_2\,v_{xx}\hspace{0.8mm}+\theta \,v_x+v(\sigma_2-c_{23}\,\sigma_3\,c_{33}^{-1}-c_{21}\,u-c_{22}\,v)\leq0, \quad x\in\mathbb{R}.
\end{cases}
\end{equation}
Because of $\mathbf{[A1]}$ and $\mathbf{[A2]}$, we can apply $(I)$ of Theorem~\ref{thm: lb<q(x)<ub} to \eqref{eqn: nonexistence diff ineq <0}. Indeed, $\mathbf{[A1]}$ assures the positivity of $\sigma_1-c_{13}\,\sigma_3\,c_{33}^{-1}$ and $\sigma_2-c_{23}\,\sigma_3\,c_{33}^{-1}$, whereas the assumption of strong competition $\mathbf{[S]}$ or the assumption of weak competition $\mathbf{[W]}$ for the nonlinearity in \eqref{eqn: nonexistence diff ineq <0} follows from $\mathbf{[A2]}$. Consequently, $(I)$ of Theorem~\ref{thm: lb<q(x)<ub} gives us a lower bound of $c_{31}\,u(x)+c_{32}\,v(x)$, i.e. for $x\in\mathbb{R}$,
\begin{equation}
c_{31}\,u(x)+c_{32}\,v(x)\geq c_{33}^{-1}\,\min\bigg[c_{31}\,d_1\min\Big[\frac{\displaystyle\Sigma_1}{\displaystyle c_{11}},\frac{\displaystyle\Sigma_2}{\displaystyle c_{21}}\Big],c_{32}\,d_2\min\Big[\frac{\displaystyle\Sigma_2}{\displaystyle c_{22}},\frac{\displaystyle\Sigma_1}{\displaystyle c_{12}}\Big]\bigg]\min\Big[\frac{\displaystyle d_1}{\displaystyle d_2},\frac{\displaystyle d_2}{\displaystyle d_1}\Big].
\end{equation}
The condition $\mathbf{[A3]}$ then yields
\begin{equation}
c_{31}\,u(x)+c_{32}\,v(x)\geq\sigma_3,\;x\in\mathbb{R},
\end{equation}
which contradicts \eqref{eqn: w(???) ???>0}. This completes the proof of the theorem.

\end{proof}

\vspace{5mm}

\textbf{Acknowledgments.} The author would like to express gratitude to Dr. Tom Mollee for his careful reading of the manuscript and valuable comments to improve the readability of the paper. The author is also grateful for inspiring discussions with and constructive suggestions from Prof. Chiun-Chuan Chen.







\vspace{5mm}











\begin{thebibliography}{99}

\bibitem{Adamson12SpeciesCyclicCompetition}
{\sc M.~W. Adamson and A.~Y. Morozov}, {\em Revising the role of species
  mobility in maintaining biodiversity in communities with cyclic competition},
  Bull. Math. Biol., 74 (2012), pp.~2004--2031.

\bibitem{Armstrong80Competitive-exclusion}
{\sc R.~A. Armstrong and R.~McGehee}, {\em Competitive exclusion}, Amer.
  Natur., 115 (1980), pp.~151--170.

\bibitem{Baczkowski98Generalized-diversity-index}
{\sc A.~J. Baczkowski, D.~N. Joanes, and G.~M. Shamia}, {\em Range of validity
  of {$\alpha$} and {$\beta$} for a generalized diversity index
  {$H(\alpha,\beta)$} due to {G}ood}, Math. Biosci., 148 (1998), pp.~115--128.

\bibitem{Berestycki09Species-shifting-climate}
{\sc H.~Berestycki, O.~Diekmann, C.~J. Nagelkerke, and P.~A. Zegeling}, {\em
  Can a species keep pace with a shifting climate?}, Bull. Math. Biol., 71
  (2009), pp.~399--429.

\bibitem{CantrellWard97Competition-mediatedCoexistence}
{\sc R.~S. Cantrell and J.~R. Ward, Jr.}, {\em On competition-mediated
  coexistence}, SIAM J. Appl. Math., 57 (1997), pp.~1311--1327.

\bibitem{Chen&Hung15Nonexistence}
{\sc C.-C. Chen and L.-C. Hung}, {\em Nonexistence of traveling wave solutions,
  exact and semi-exact traveling wave solutions for diffusive lotka-volterra
  systems of three competing species}, Communications on Pure and Applied
  Analysis, To appear.

\bibitem{CHMU-semi}
{\sc C.-C. Chen, L.-C. Hung, M.~Mimura, M.~Tohma, and D.~Ueyama}, {\em
  Semi-exact equilibrium solutions for three-species competition-diffusion
  systems}, Hiroshima Math J., 43 (2013), pp.~176--206.

\bibitem{CHMU}
{\sc C.-C. Chen, L.-C. Hung, M.~Mimura, and D.~Ueyama}, {\em Exact travelling
  wave solutions of three-species competition-diffusion systems}, Discrete
  Contin. Dyn. Syst. Ser. B, 17 (2012), pp.~2653--2669.

\bibitem{demottoni79}
{\sc P.~de~Mottoni}, {\em {Qualitative analysis for some quasilinear parabolic
  systems}}, Institute of Math., Polish Academy Sci., zam, 11 (1979), p.~190.

\bibitem{EiMimuraIkota99SegregatingCompetition-diffusion}
{\sc S.-I. Ei, R.~Ikota, and M.~Mimura}, {\em Segregating partition problem in
  competition-diffusion systems}, Interfaces Free Bound., 1 (1999), pp.~57--80.

\bibitem{Good53Estimation-population-parameters}
{\sc I.~J. Good}, {\em The population frequencies of species and the estimation
  of population parameters}, Biometrika, 40 (1953), pp.~237--264.

\bibitem{Grossberg78Decision-Patterns-Oscillations-LVcompetitive}
{\sc S.~Grossberg}, {\em Decisions, patterns, and oscillations in nonlinear
  competitve systems with applications to {V}olterra-{L}otka systems}, J.
  Theoret. Biol., 73 (1978), pp.~101--130.

\bibitem{Gyllenberg09LV3species-Heteroclinic}
{\sc M.~Gyllenberg and P.~Yan}, {\em On a conjecture for three-dimensional
  competitive {L}otka-{V}olterra systems with a heteroclinic cycle}, Differ.
  Equ. Appl., 1 (2009), pp.~473--490.

\bibitem{Hallam79Persistence-Extinction3speciesLV}
{\sc T.~G. Hallam, L.~J. Svoboda, and T.~C. Gard}, {\em Persistence and
  extinction in three species {L}otka-{V}olterra competitive systems}, Math.
  Biosci., 46 (1979), pp.~117--124.

\bibitem{Hirsch83Strongly-monotone-semiflows}
{\sc M.~W. Hirsch}, {\em Differential equations and convergence almost
  everywhere in strongly monotone semiflows}, Contemp. Math, 17 (1983),
  pp.~267--285.

\bibitem{Hou&Leung08}
{\sc X.~Hou and A.~W. Leung}, {\em Traveling wave solutions for a competitive
  reaction-diffusion system and their asymptotics}, Nonlinear Anal. Real World
  Appl., 9 (2008), pp.~2196--2213.

\bibitem{Hsu08Competitive-exclusion}
{\sc S.-B. Hsu and T.-H. Hsu}, {\em Competitive exclusion of microbial species
  for a single nutrient with internal storage}, SIAM J. Appl. Math., 68 (2008),
  pp.~1600--1617.

\bibitem{Hsu-Smith-Waltman96Competitive-exclusion-coexistence-Competitive}
{\sc S.~B. Hsu, H.~L. Smith, and P.~Waltman}, {\em Competitive exclusion and
  coexistence for competitive systems on ordered {B}anach spaces}, Trans. Amer.
  Math. Soc., 348 (1996), pp.~4083--4094.

\bibitem{hung2012JJIAM}
{\sc L.-C. Hung}, {\em Exact traveling wave solutions for diffusive
  {L}otka-{V}olterra systems of two competing species}, Jpn. J. Ind. Appl.
  Math., 29 (2012), pp.~237--251.

\bibitem{Jang13Competitive-exclusion-Leslie-Gower-competition-Allee}
{\sc S.~R.-J. Jang}, {\em Competitive exclusion and coexistence in a
  {L}eslie-{G}ower competition model with {A}llee effects}, Appl. Anal., 92
  (2013), pp.~1527--1540.

\bibitem{Kanel06}
{\sc J.~I. Kanel}, {\em On the wave front solution of a competition-diffusion
  system in population dynamics}, Nonlinear Anal., 65 (2006), pp.~301--320.

\bibitem{Kanel&Zhou96estimatespeed}
{\sc J.~I. Kanel and L.~Zhou}, {\em Existence of wave front solutions and
  estimates of wave speed for a competition-diffusion system}, Nonlinear Anal.,
  27 (1996), pp.~579--587.

\bibitem{Kastendiek82Competitor-mediatedCoexistence3Species}
{\sc J.~Kastendiek}, {\em Competitor-mediated coexistence: interactions among
  three species of benthic macroalgae}, Journal of Experimental Marine Biology
  and Ecology, 62 (1982), pp.~201--210.

\bibitem{KishimotoWeinberger85Stable-equilibriaConvex}
{\sc K.~Kishimoto and H.~F. Weinberger}, {\em The spatial homogeneity of stable
  equilibria of some reaction-diffusion systems on convex domains}, J.
  Differential Equations, 58 (1985), pp.~15--21.

\bibitem{KoRyuAhn14Coexistence3Competing-species}
{\sc W.~Ko, K.~Ryu, and I.~Ahn}, {\em Coexistence of three competing species
  with non-negative cross-diffusion rate}, J. Dyn. Control Syst., 20 (2014),
  pp.~229--240.

\bibitem{KochMedina95longtimebehv}
{\sc P.~Koch~Medina and G.~Sch{\"a}tti}, {\em Long-time behaviour for
  reaction-diffusion equations on {$\bold R\sp N$}}, Nonlinear Anal., 25
  (1995), pp.~831--870.

\bibitem{Maier13Integration3-dimensionalLV}
{\sc R.~S. Maier}, {\em The integration of three-dimensional {L}otka-{V}olterra
  systems}, Proc. R. Soc. Lond. Ser. A Math. Phys. Eng. Sci., 469 (2013),
  pp.~20120693, 27.

\bibitem{McGehee77Competitive-exclusion}
{\sc R.~McGehee and R.~A. Armstrong}, {\em Some mathematical problems
  concerning the ecological principle of competitive exclusion}, J.
  Differential Equations, 23 (1977), pp.~30--52.

\bibitem{Mimura15DynamicCoexistence3species}
{\sc M.~Mimura and M.~Tohma}, {\em Dynamic coexistence in a three-species
  competition--diffusion system}, Ecological Complexity, 21 (2015),
  pp.~215--232.

\bibitem{PetrovskiiShigesada01Spatio-temporalThree-competitive-species}
{\sc S.~Petrovskii, K.~Kawasaki, F.~Takasu, and N.~Shigesada}, {\em Diffusive
  waves, dynamical stabilization and spatio-temporal chaos in a community of
  three competitive species}, Japan J. Indust. Appl. Math., 18 (2001),
  pp.~459--481.

\bibitem{Ramezan11Shannon-diversity-index}
{\sc H.~Ramezani and S.~Holm}, {\em Sample based estimation of landscape
  metrics; accuracy of line intersect sampling for estimating edge density and
  {S}hannon's diversity index}, Environ. Ecol. Stat., 18 (2011), pp.~109--130.

\bibitem{Sanchez00Nonautonomous-Fisher}
{\sc L.~Sanchez}, {\em A note on a nonautonomous {O}.{D}.{E}. related to the
  {F}isher equation}, J. Comput. Appl. Math., 113 (2000), pp.~201--209.

\bibitem{Simpson49Measurement-diversity}
{\sc E.~H. Simpson}, {\em Measurement of diversity.}, Nature,  (1949).

\bibitem{Smith94Competition}
{\sc H.~L. Smith and P.~Waltman}, {\em Competition for a single limiting
  resource in continuous culture: the variable-yield model}, SIAM J. Appl.
  Math., 54 (1994), pp.~1113--1131.

\bibitem{Zeeman98Three-dimensionalCompetitiveLV}
{\sc P.~van~den Driessche and M.~L. Zeeman}, {\em Three-dimensional competitive
  {L}otka-{V}olterra systems with no periodic orbits}, SIAM J. Appl. Math., 58
  (1998), pp.~227--234.

\bibitem{Volperts94TWS-Parabolic}
{\sc A.~I. Volpert, V.~A. Volpert, and V.~A. Volpert}, {\em Traveling wave
  solutions of parabolic systems}, vol.~140 of Translations of Mathematical
  Monographs, American Mathematical Society, Providence, RI, 1994.
\newblock Translated from the Russian manuscript by James F. Heyda.

\bibitem{Volpert99Non-autonomous-KPP}
{\sc V.~A. Volpert and Y.~M. Suhov}, {\em Stationary solutions of
  non-autonomous {K}olmogorov-{P}etrovsky-{P}iskunov equations}, Ergodic Theory
  Dynam. Systems, 19 (1999), pp.~809--835.

\bibitem{Zeeman93Hopf-bifurcationsCompetitive3speciesLV}
{\sc M.~L. Zeeman}, {\em Hopf bifurcations in competitive three-dimensional
  {L}otka-{V}olterra systems}, Dynam. Stability Systems, 8 (1993),
  pp.~189--217.

\end{thebibliography}

\end{document}